\documentclass{amsart}

\usepackage{amssymb}
\usepackage[all]{xy}
\usepackage[colorlinks, citecolor=red, linkcolor=red]{hyperref}
\def\ZZ{ \mathbb{Z} }
\def\StrCat{ \mathbf{StrCat} }
\def\Mod{ \mathrm{Mod} }
\def\uHom{ \underline{\mathrm{Hom}} }

\def\Set{ \mathrm{Set} }

\def\inc{ \hookrightarrow }
\def\ito{ \rightarrowtail }

\def\dto{ \rightrightarrows }

\def\lrto{ \leftrightarrows }

\def\weq{ \overset{\sim}{\longrightarrow} }
\def\eps{ \epsilon }

\def\sg{ \sigma }

\def\Ee{\mathcal{E}}
\def\uEe{\underline{\Ee}}
\def\op{\mathrm{op}}
\def\Ho{\mathbf{Ho}}
\def\sg{\sigma}
\def\Aa{\mathcal{A}}
\def\uAa{\underline{\Aa}}
\def\Bb{\mathcal{B}}
\def\uBb{\underline{\Bb}}

\def\Cc{\mathcal{C}}

\def\Ab{\mathrm{Ab}}
\def\Cyl{\mathrm{Cyl}}
\def\GS{\mathcal{GS}}
\def\uA{\underline{A}}

\def\Sp{\mathrm{Sp}}

\def\cleft{\hbox{[\kern-.16em\hbox{[}}}
\def\cright{\hbox{]\kern-.16em\hbox{]}}}

\newdir{ >}{{}*!/-7pt/@{>}}

\theoremstyle{plain}
\newtheorem{defn}{Definition}[section]
\newtheorem{thm}[defn]{Theorem}
\newtheorem{lem}[defn]{Lemma}
\newtheorem{prop}[defn]{Proposition}
\newtheorem{cor}[defn]{Corollary}

\theoremstyle{remark}
\newtheorem{rem}[defn]{Remark}

\title{Gabriel-Morita theory for excisive model categories}

\author{Clemens Berger}
\address{Universit\'{e} de Nice, Lab. J. A. Dieudonn\'{e}, Parc Valrose, 06108 Nice Cedex, France}
\email{cberger@math.unice.fr}

\author{Kruna Ratkovic}
\address{University of Donja Gorica, Faculty of Applied Sciences, Donja Gorica, 81000 Podgorica, Montenegro}
\email{kruna.ratkovic@udg.edu.me}

\date{October 20, 2017}

\newcounter{diagram}[section]
\renewcommand\thediagram{\thesection.\arabic{diagram}}
\newenvironment{diagram}[1]{
\refstepcounter{diagram}
\def\titredudiagramme{#1}
\vspace{4mm}
}{
\begin{center}
  {\rm Diagram \thediagram. \emph{\titredudiagramme}}
\end{center}
}

\begin{document}

\begin{abstract}We develop a Gabriel-Morita theory for strong monads on pointed monoidal model categories. Assuming that the model category is excisive, i.e. the derived suspension functor is conservative, we show that if the monad $T$ preserves cofibre sequences up to homotopy and has a weakly invertible strength, then the category of $T$-algebras is Quillen equivalent to the category of $T(I)$-modules where $I$ is the monoidal unit. This recovers Schwede's theorem on connective stable homotopy over a pointed Lawvere theory as special case.\end{abstract}

\maketitle

\section*{Introduction}\label{Intro}

Morita \cite{MorDM} gave a precise criterion for when two rings have equivalent categories of modules. His criterion may be derived from Gabriel's characterisation \cite{GDCA} of those abelian categories which are equivalent to categories of $R$-modules for a unital associative ring $R$. In the present text we give a homotopical version of the latter in the framework of monoidal model categories, cf. Hovey \cite{HovMC} and Quillen \cite{QuiHA}.

We present a class of enriched monads $T$ acting on pointed monoidal model categories $\Ee$ with the property that the category $\Ee^T$ of $T$-algebras is Quillen equivalent to the category $\Mod_{T(I)}$ of $T(I)$-modules for a functorially associated monoid $T(I)$. The underlying object of this monoid is the free $T$-algebra on the monoidal unit $I$ of $\Ee$. The monoid structure is obtained through identification of $T(I)$ with the enriched endomorphism object $\uEe^T(T(I),T(I))$ in much the same way as any ring $R$ may be identified with the endomorphism ring of a free $R$-module of rank one.

It turns out that in our homotopical setting Kock's correspondence \cite{KockMS, KockSF} between strong and enriched monads is very helpful. In particular, the \emph{strength} of an enriched monad $T$ yields a direct formula for the monoid structure of $T(I)$ and moreover induces a natural transformation of monads $\lambda: -\otimes T(I)\to T$ relating the categories of $T(I)$-modules and of $T$-algebras by a canonical adjunction. In the special case of an enrichment in abelian groups this adjunction contains classical Gabriel-Morita theory in an embryonic form. Our homotopical version thereof essentially consists of finding suitable homotopical hypotheses on $\Ee$ and on $T$ in order to transform the adjunction into a Quillen equivalence of model categories.

We show in our main Theorem \ref{mainexcisive} that (apart from a few technical assumptions) the adjunction $\lambda_!:\Mod_{T(I)}\lrto\Ee^T:\lambda^*$ is a Quillen equivalence whenever the pointed model category $\Ee$ is excisive, the strong monad $T:\Ee\to\Ee$ preserves cofibre sequences up to homotopy, and the strength $\sg_{X,Y}:X\otimes T(Y)\to T(X\otimes Y)$ is a weak equivalence for all cofibrant objects $X,Y$ of $\Ee$.

By an \emph{excisive} model category we mean a left proper pointed model category in which the derived suspension functor is \emph{conservative}, i.e. reflects isomorphisms. This captures roughly speaking half of the well-known properties of a \emph{stable} model category. Yet, there are important examples of excisive model categories which are not stable, as for instance the category of based topological spaces localised with respect to a generalised homology theory, or the category of nonnegatively graded $R$-chain complexes, or Segal's category of $\Gamma$-spaces \cite{SECCT} equipped with the stable model structure of Bousfield-Friedlander \cite{BouHT}. It is surprising how much homological algebra can be deduced just from a conservative derived suspension functor. In contrast to the first two examples, the derived suspension functor of the third example is even \emph{fully faithful} in which case we call the model category \emph{stably excisive}.

Homotopy pushouts in an excisive model category $\Ee$ are characterised by having weakly equivalent parallel homotopy cofibres. In particular, a monad $T$ on $\Ee$ which preserves cofibre sequences up to homotopy actually preserves all homotopy pushouts in $\Ee$. This will be enough to show that the forgetful functor takes free cell attachments of $T$-algebras to homotopy pushouts in $\Ee$, a key property in our homotopical Gabriel-Morita theory. On the way we make essential use of the so-called \emph{bar resolution} of a $T$-algebra which under suitable conditions yields a \emph{cofibrant replacement}. Segal's \emph{fat realisation} \cite{SECCT} is a most useful device in this context.

In the special case where $\Ee$ is the stably excisive model category of $\Gamma$-spaces, and $T$ is the monad on $\Gamma$-spaces induced by a pointed simplicial Lawvere theory \cite{Lawvere}, we recover Schwede's theorem \cite{SchSHAT} on connective stable homotopy over a pointed Lawvere theory. To understand the latter from the perspective of homotopical Gabriel-Morita theory has been one of main motivations of the present text.

The article is organised as follows:\vspace{1ex}

Section \ref{CEbSCEC} reviews Kock's correspondence between strong and enriched monads using Street's formal theory of monads \cite{KockMS,KockSF,StrFTM}. We give a ``monadic'' proof of Gabriel's characterisation  \cite{GDCA} of module categories among abelian categories. 

Section \ref{HMT} introduces excisive model categories with emphasis on the homotopical diagram lemmas available in such categories. We also discuss the relationship with Goodwillie's notion of excisive identity functor \cite{Good}.

Section \ref{Mth} presents our homotopical Gabriel-Morita theory giving sufficient conditions for the adjunction $\lambda_!:\Mod_{T(I)}\lrto\Ee^T:\lambda^*$ to be a Quillen equivalence. As an application we give a quick proof of Schwede's theorem on connective stable homotopy theory for models of a pointed simplicial Lawvere theory \cite{Lawvere,SchSHAT}.

Section \ref{taubar} studies the bar resolution of $T$-algebras, used in Section \ref{Mth} in order to show that certain homotopical properties of a monad $T$ are inherited by the forgetful functor $\Ee^T\to\Ee$. Our main tool is a new cofibrancy notion for simplicial objets of a model category, interpolationg between degreewise and Reedy cofibrancy, cf. \cite{SECCT}.\vspace{1ex}

{\bf Terminology.} We assume the reader is familiar with enriched category theory and homotopical algebra and follow the terminology of Borceux \cite{BorHCA} for the former and the terminology of Quillen \cite{QuiHA}, Hovey \cite{HovMC} and Hirschhorn \cite{Hi} for the latter.

Cofibrations will be denoted by $X\ito Y$, weak equivalences by $X\weq Y$. An \emph{acyclic (co)fibration} is a morphism which is simultaneously a (co)fibration and a weak equivalence. An object of a pointed model category is called \emph{acyclic} if it is weakly equivalent to a null-object. An object of a model category is called \emph{cofibrant} (resp. \emph{fibrant}) if the unique map from an initial object (resp. to a terminal object) is a cofibration (resp. fibration).

\section{Strong vs enriched monads}\label{CEbSCEC}

In this section, we recall Kock's correspondence \cite{KockMS,KockSF} between \emph{strong} and \emph{enriched} monads on a closed symmetric monoidal category $(\Ee,\otimes, I)$. We deduce this correspondence from Street's formal theory of monads \cite{StrFTM} showing that tensored $\Ee$-categories are the $0$-cells of two different but $2$-isomorphic $2$-categories, one with strong functors and strong natural transformations as $1$- and $2$-cells, the other with $\Ee$-functors and $\Ee$-natural transformations as $1$- and $2$-cells.

Our main interest comes from the resulting equivalence between the category of \emph{monoids} in $\Ee$ and the category of \emph{linear monads} on $\Ee$, i.e. those strong monads which have invertible strength. Each strong monad $T$ comes equipped with a linear approximation $\lambda:-\otimes T(I)\to T$ which is invertible precisely when $T$ is linear.

\subsection{Strong functors and strong natural transformations}-\label{Str}\label{deffcomp}\label{defstnatt}\vspace{1ex}

We fix once and for all a closed symmetric monoidal category $\Ee=(\Ee,\otimes,I)$. The internal hom of $\Ee$ will be denoted $\uEe(-,-)$ while the $\Ee$-valued hom of an $\Ee$-category $\Aa$ will be denoted $\uAa(-,-)$. Recall that an $\Ee$-category $\Aa$ is called \emph{tensored} (cf. Quillen \cite{QuiHA}) if for each object $A$ in $\Aa$, the $\Ee$-functor $\underline{\Aa}(A,-):\Aa\rightarrow\Ee$ admits an $\Ee$-enriched left adjoint. These left adjoints assemble into a \emph{left $\Ee$-action} on $\Aa$ which we shall denote as an \emph{external tensor product} $-\otimes_\Ee-:\Ee\times\Aa\rightarrow\Aa$. With this notation the individual adjunctions give rise to trinatural $\Ee$-isomorphisms$$\underline{\Aa}(X\otimes_\Ee A,B)\cong\underline{\Ee}\left(X,\underline{\Aa}(A,B)\right)$$for any objects $A,B$ in $\Aa$ and $X$ in $\Ee$. In particular, we get canonical isomorphisms $l_A:I\otimes_\Ee A\cong A$ and  $a_{X,Y,A}:(X\otimes Y)\otimes_\Ee A\cong X\otimes_\Ee(Y\otimes_\Ee A)$ for any objects $X,Y$ in $\Ee$ and $A$ in $\Aa$.

A symmetric monoidal category $(\Ee,\otimes,I)$ is \emph{closed} if and only if $\Ee$ is a tensored $\Ee$-category such that the left $\Ee$-action agrees with the monoidal structure of $\Ee$.

A \emph{strong functor} $(F,\sg)$ between tensored $\Ee$-categories is a functor $F:\Aa\to\Bb$ endowed with a \emph{strength}%\footnote{Kock \cite{KockSF} distinguishes for functors $F$ strength $\uAa(A,A')\to\uBb(F(A),F(A'))$ and tensorial strength $X\otimes_\Ee F(A)\to F(X\otimes_\Ee A)$. We call the former an \emph{enrichment} and the latter a \emph{strength}.}
$\sg_{X,A}: X\otimes_\Ee F(A) \rightarrow F(X\otimes_\Ee A)$ natural in $X$ and $A$ and such that the following two diagrams in $\Bb$ commute:\begin{diagram}{Unit constraint for the strength}
$$
\xymatrix @!0 @C=3cm @R=1.5cm{\relax
 I\otimes_\Ee F(A)\ar[rr]^{\sg_{I,A}} \ar[rd]_{l_{F(A)}} && F(I\otimes_\Ee A) \ar[ld]^{F(l_A)} \\
& F(A)
}
$$
\end{diagram}
\begin{diagram}{Associativity constraint for the strength}%\label{diagassbeta}
$$
\xymatrix @!0 @C=5cm @R=2.5cm{\relax
 (X\otimes Y)\otimes_\Ee F(A) \ar[r]^{a_{X,Y,F(A)}} \ar[d]_{\sg_{X\otimes Y,A}} & X\otimes_\Ee(Y\otimes_\Ee F(A)) \ar[r]^{X\otimes_\Ee\sg_{Y,A}}  & X\otimes_\Ee F(Y\otimes_\Ee A) \ar[d]^{\sg_{X,Y\otimes_\Ee A}} \\
F((X\otimes Y)\otimes_\Ee A)\ar[rr]^{F(a_{X,Y,A})} && F(X\otimes_\Ee (Y\otimes_\Ee A))
}
$$
\end{diagram}
Strong functors $(F_1,\sg_1):\Aa\to\Bb$ and $(F_2,\sg_2):\Bb\to\Cc$ compose so as to give a strong functor $(F_2F_1,\sg):\Aa\to\Cc$ with strength
$$
\xymatrix @!0 @C=3cm @R=1.5cm{\relax
 X\otimes_\Ee F_2(F_1(A)) \ar[rr]^{\sg} \ar[rd]_{\sg_2} && F_2(F_1(X\otimes_\Ee A))  \\
&  F_2(X\otimes_\Ee F_1(A)) \ar[ru]_{{F_2}(\sg_1)}
}
$$

A \emph{strong natural transformation} $\psi$ between strong functors $(F,\sg^F),(G,\sg^G):\Aa\rightarrow \Bb$ is a natural transformation $\psi:F\Rightarrow G$ such that the following diagram
$$
\xymatrix @C=3cm @R=2.5cm{\relax
X\otimes_\Ee F(A) \ar[r]^{\sg^F_{X,A}} \ar[d]_{X\otimes_\Ee \psi_A} & F(X\otimes_\Ee A) \ar[d]^{\psi_{X\otimes_\Ee A}}\\
X\otimes_\Ee G(A) \ar[r]_{\sg^G_{X,A}} & G(X\otimes_\Ee A)
}
$$commutes for all objects $X$ in $\Ee$ and all objects $A$ in $\Aa$.

An adjunction $F:\Aa\lrto\Bb:G$ between tensored $\Ee$-categories $\Aa,\Bb$ is called \emph{strong} if $F$ and $G$ are strong functors, and unit $\eta:id_\Aa\Rightarrow GF$ and counit $\eps:FG\Rightarrow id_\Bb$ are strong natural transformations.

% We place now in a $2$-categorical framework (cf. \cite{KelBCECT,KelRE2C}).
\begin{lem}Tensored $\Ee$-enriched categories, strong functors and strong natural transformations constitute a $2$-category, denoted $\StrCat_\Ee$.
\end{lem}

\begin{proof}Straightforward.\end{proof}

\begin{defn}[cf. Kock \cite{KockMS,KockSF}, Street \cite{StrFTM}, Moggi \cite{MogNCM}]--\label{defStMon}

A \emph{strong monad} $(T,\mu,\eta,\sg)$ on a closed symmetric monoidal category $(\Ee,\otimes,I)$ is a monad in the $2$-category $\StrCat_\Ee$, i.e. it consists of a \emph{monad} $\left(T,\mu,\eta\right)$ on $\Ee$ endowed with a binatural \emph{strength} $\sg_{X,Y}: X\otimes T(Y)\rightarrow T(X\otimes Y)$ such that the following four diagrams commute for all objects $X,Y,Z$ in $\Ee$:
\begin{diagram}{Unit constraint for the strength}
$$
\xymatrix @!0 @C=3cm @R=1.5cm{\relax
 I\otimes T(X)\ar[rr]^{\sg_{I,X}} \ar[rd]_{l_{T(X)}} && T(I\otimes X) \ar[ld]^{T(l_X)} \\
& T(X)
}
$$
\end{diagram}
\begin{diagram}{Associativity constraint for the strength}\label{diagassbeta}
$$
\xymatrix @!0 @C=5cm @R=2.5cm{\relax
 (X\otimes Y)\otimes T(Z) \ar[r]^{a_{X,Y,T(Z)}} \ar[d]_{\sg_{X\otimes Y,Z}} & X\otimes(Y\otimes T(Z)) \ar[r]^{X\otimes\sg_{Y,Z}}  & X\otimes T(Y\otimes Z) \ar[d]^{\sg_{X,Y\otimes Z}} \\
T((X\otimes Y)\otimes Z)\ar[rr]^{T(a_{X,Y,Z})} && T(X\otimes (Y\otimes Z))
}
$$
\end{diagram}
\begin{diagram}{Strong naturality of the unit}\label{diagsteta}
$$
\xymatrix @!0 @C=3cm @R=1.5cm{\relax
 X\otimes T(Y) \ar[rr]^{\sg_{X,Y}} && T(X\otimes Y)  \\
& X\otimes Y \ar[lu]^{X\otimes \eta_{Y}} \ar[ru]_{\eta_{X\otimes Y}}
}
$$
\end{diagram}\label{diagstmu}
\begin{diagram}{Strong naturality of the multiplication}
$$
\xymatrix @!0 @C=4.7cm @R=2.3cm{\relax
X\otimes T^{2}(Y) \ar[r]^{\sg_{X,T(Y)}} \ar[d]_{X\otimes \mu_{Y}} & T(X\otimes T(Y)) \ar[r]^{T(\sg_{X,Y})} & T^{2}(X\otimes Y) \ar[d]^{\mu_{X\otimes Y}}\\
X\otimes T(Y) \ar[rr]^{\sg_{X,Y}} && T(X\otimes Y)
}
$$
\end{diagram}
\end{defn}

\begin{rem}The reader may observe that in this unravelled definition of a strong monad the \emph{closed} structure of $\Ee$ has disappeared. It was one of the insights of Kock \cite{KockMS,KockSF} that a strong monad $T$ is definable independently of any existing closed structure, yet representing a monad which \emph{implicitly} takes care of any self-enrichment of $\Ee$ thanks to the adjoint relationship between tensor and hom.

This point of view will be even more important when dealing with strong monads acting on monoidal model categories $\Ee$ since it turns out that the interference between the strength of $T$ and the \emph{cofibration structure} of $\Ee$ is more transparent than the interference between the enrichment of $T$ and the \emph{fibration structure} of $\Ee$.\end{rem}

\begin{prop}[cf. Kock \cite{KockSF}, Theorem 1.3]\label{th2-cat}--

For any closed symmetric monoidal category $\Ee$, the $2$-category of tensored $\Ee$-categories, strong functors and strong natural transformations is $2$-isomorphic to the $2$-category of tensored $\Ee$-categories, $\Ee$-functors and $\Ee$-natural transformations.
\end{prop}

\begin{proof}The $0$-cells of the two $2$-categories are the same. We shall indicate canonical bijections between the respective $1$- and $2$-cells and leave it to the reader to check compatibility with the $2$-categorical structures on both sides (see also \cite{RatMTEC}).

Note first that there is an evaluation map$$ev_{A_1,A_2}:\uAa(A_1,A_2)\otimes_\Ee A_1\to A_2$$adjoint to the identity of $\uAa(A_1,A_2)$, as well as a coevaluation map$$coev_{X,A}:X\to\uAa(A,X\otimes_\Ee A)$$adjoint to the identity of $X\otimes_\Ee A$.

For a given functor $F:\Aa\to\Bb$, any strength $\sg_{X,A}:X\otimes_\Ee F(A)\to F(X\otimes_\Ee A)$ determines an enrichment $\varphi_{A_1,A_2}$ by taking the adjoint of the composite map$$\uAa(A_1,A_2)\otimes_\Ee F(A_1)\overset{\sg_{\uAa(A_1,A_2),A_1}}{\longrightarrow}F(\uAa(A_1,A_2)\otimes_\Ee A_1)\overset{F(ev_{A_1,A_2})}{\longrightarrow}F(A_2).$$

Conversely, any enrichment $\varphi_{A_1,A_2}:\uAa(A_1,A_2)\longrightarrow\uBb(F(A_1),F(A_2))$ determines a strength $\sg_{X,A}$ by taking the adjoint of the composite map$$X\overset{coev_{X,A}}{\longrightarrow}\uAa(A,X\otimes_\Ee A)\overset{\varphi_{A,X\otimes_\Ee A}}{\longrightarrow}\uBb(F(A),F(X\otimes_\Ee A)).$$ The two assignments are mutually inverse, because evaluation and coevaluation satisfy the familiar triangle relations of counit and unit of an adjunction.

Let $\psi:F\Rightarrow G$ be a natural transformation between functors $F,G:\Aa\to\Bb$ which are endowed with strengths $\sg^F,\sg^G$, resp. their corresponding enrichments $\varphi^F,\varphi^G$. Assume that $\psi$ is strong natural. Then the following diagram
$$
\xymatrix @!0 @C=4.5cm @R=2cm{\relax
 \uAa(A_1,A_2)\otimes_\Ee F(A_1) \ar[r]^{\sg^F} \ar[d]_{\uAa(A_1,A_2)}^{\otimes_\Ee\psi_{A_1}} &  F(\uAa(A_1,A_2)\otimes_\Ee A_1) \ar[r]^{F(ev_{A_1,A_2})} \ar[d]_{\psi_{\uAa(A_1,A_2)\otimes_\Ee A_1}} & \ar[d]_{\psi_{A_2}} F(A_2)\\
 \uAa(A_1,A_2)\otimes_\Ee G(A_1) \ar[r]_{\sg^G} &  G(\uAa(A_1,A_2)\otimes_\Ee A_1) \ar[r]_{G(ev_{A_1,A_2})} & G(A_2)
}
$$commutes, and a suitable adjoint of the outer rectangle shows that $\psi$ is $\Ee$-natural. Conversely, assume that $\psi$ is $\Ee$-natural. Then the following diagram
$$
\xymatrix @!0 @C=4.5cm @R=3cm{\relax
X\ar[r]^{coev_{X,A}}  \ar[rd]_{\hat{\sg}^F}  \ar@/^4pc/[rr]^{\hat{\sg}^G} & \uAa(A,X\otimes_\Ee A)\ar[d]^{\varphi^F} \ar@{}[rd] \ar[r]^{\varphi^G} & \uBb(G(A),G(X\otimes_\Ee A))\ar[d]^{\psi^*}\\
& \uBb(F(A),F(X\otimes_\Ee A)) \ar[r]_{\psi_*} & \uBb(F(A),G(X\otimes_\Ee A))
}
$$commutes which establishes by adjunction that $\psi$ is strong natural.\end{proof}

\begin{cor}\label{cormons-e}\label{proptens}For any closed symmetric monoidal category $\Ee$, there is a canonical one-to-one correspondence between strong monads and $\Ee$-enriched monads on $\Ee$.

If $\Ee$ has equalisers, then the category $\Ee^T$ of $T$-algebras for a strong monad $T$ is $\Ee$-enriched. If in addition $\Ee^T$ has reflexive coequalisers, then $\Ee^T$ is a tensored $\Ee$-category and the free-forgetful adjunction $F_T:\Ee\lrto\Ee^T:U_T$ is a strong adjunction.\end{cor}

\begin{proof}The first assertion follows from Proposition \ref{th2-cat} and Street's formal theory of monads \cite{StrFTM}. The second assertion is classic for $\Ee$-enriched monads $T$ and produces an $\Ee$-enriched adjunction $F_T:\Ee\lrto\Ee^T:U_T$. For the third assertion, it suffices (again by Proposition \ref{th2-cat} and \cite{StrFTM}) to establish the existence of $\Ee$-tensors in $\Ee^T$.

For a $T$-algebra $(X,\xi_X:T(X)\to X)$ in $\Ee^T$ and an object $Z$ in $\Ee$ we define the $\Ee$-tensor $Z\otimes_\Ee (X,\xi_X)$ in $\Ee^T$ by the following reflexive coequaliser

$$
\xymatrix @!0 @C=3cm @R=1.8cm{\relax
T(Z\otimes T(X)) \ar[rr]^{T(Z\otimes \xi_{X}) } \ar[rd]_{T(\sg_{Z,X})} && T(Z\otimes X) \ar[r] & Z\otimes_\Ee(X,\xi_X)\\
& T^2(Z\otimes X) \ar[ru]_{\mu_{Z\otimes X}}
}
$$ in $\Ee^T$. Note that the $T$-algebra maps $\mu_{Z\otimes X}\circ T(\sg_{Z,X})$ and $T(Z\otimes\xi_X)$ admit $T(Z\otimes\eta_X)$ as common section because of the strong naturality of the unit $\eta$, cf. Diagram \ref{diagsteta}. For $T$-algebras $(X,\xi_X),(Y,\xi_Y)$ and any object $Z$ in $\Ee$ we get the following chain of trinatural $\Ee$-isomorphisms
\begin{align*}\uEe^T(Z\otimes_\Ee X,Y)&\cong\uEe^T(Coeq\lbrace T(Z\otimes T(X))\rightrightarrows T(Z\otimes X)\rbrace,Y)\\
&\cong Eq\lbrace\uEe^T(T(Z\otimes X),Y) \rightrightarrows\uEe^T(T(Z\otimes T(X)),Y)\rbrace \\
&\cong Eq\lbrace \uEe(Z\otimes X,Y) \rightrightarrows \uEe(Z\otimes T(X),Y)\rbrace\\
&\cong\uEe(Z,Eq\lbrace \uEe(X,Y)\rightrightarrows \uEe(T(X),Y)\rbrace)\\
&\cong\uEe(Z,\uEe^T(X,Y))
\end{align*}
where we have used on one side that the free-forgetful adjunction $(F_T,U_T)$ is an $\Ee$-adjunction, and on the other side that the pair $(\mu_{Z\otimes X}\circ T(\sg_{Z,X}),T(Z\otimes\xi_X))$ transforms into the pair $((\uEe(T(X),\xi_Y)\circ\varphi_{X,Y}^T,\uEe(\xi_X,Y)):\uEe(X,Y)\rightrightarrows\uEe(T(X),Y)$ the equaliser of which computes the $\Ee$-enriched hom $\uEe^T(X,Y)$.\end{proof}

\begin{rem}\label{forgetfulstrength}%The construction of $\Ee$-tensors in $\Ee^T$ under the assumption that $\Ee^T$ has reflexive coequalisers can also be derived from an $\Ee$-enriched version of Dubuc's adjoint functor lifting theorem \cite{DubucTriangle}.
It follows from an adjunction argument that the strength of the free functor $Z_1\otimes_\Ee F_T(Z_2)\to F_T(Z_1\otimes Z_2)$ is invertible for all objects $Z_1,Z_2$ of $\Ee$. However, the strength of the forgetful functor $Z\otimes U_T(X,\xi_X)\to U_T(Z\otimes_\Ee(X,\xi_X))$ is in general not invertible and is explicitly induced by the following diagram in $\Ee$
$$
\xymatrix @!0 @C=3cm @R=2cm{\relax
Z\otimes T^2(X) \ar@<2pt>[rr]^{Z\otimes \mu_{X}} \ar@<-2pt>[rr]_{Z\otimes T(\xi_{X})} \ar[d]_{\sg_{Z,T(X)}} && Z\otimes T(X) \ar[r]^{Z\otimes\xi_X} \ar[d]^{\sg_{Z,X}}&Z\otimes U_T(X,\xi_X)\ar@{.>}[d] \\
T(Z\otimes T(X)) \ar[rr]^{T(Z\otimes \xi_X) } \ar[rd]_{T(\sg_{Z,X})} && T(Z\otimes X) \ar[r] & U_T(Z\otimes_\Ee(X,\xi_X)) \\
& T^2(Z\otimes X) \ar[ru]_{\mu_{Z\otimes X}}
}
$$
where the internal left square commutes by naturality of the strength $\sg$ and the external pentagonal diagram on the left commutes by strong naturality of $\mu$.\end{rem}

\subsection{Linear monads and monoids}

A strong monad $(T,\mu,\eta,\sg)$ will be called \emph{linear} if its strength $\sg_{X,Y}: X\otimes T(Y)\to T(X\otimes Y)$ is \emph{invertible} for all $X,Y$.

Any monoid $M$ in a symmetric monoidal category $(\Ee,\otimes,I)$ induces a linear monad $T_M(X)=X\otimes M$ on $\Ee$ where the strength is induced by the associativity constraint of the monoidal structure of $\Ee$. We will see that any linear monad on $\Ee$ is of the form $T_M$ for an essentially unique monoid $M=T(I)$ in $\Ee$. In other words, the category of linear monads on $\Ee$ is equivalent to the category of monoids in $\Ee$.

\begin{lem}\label{propmonTI}
Let $(T,\mu,\eta,\sg)$ be a strong monad on a closed symmetric monoidal category $(\Ee,\otimes,I)$. The object $T(I)$ of $\Ee$ is isomorphic to the $\Ee$-enriched endomorphism monoid of the free $T$-algebra $F_T(I)$ generated by the monoidal unit $I$.

The induced monoid structure on $T(I)$ is given by$$m_{T(I)}:T(I)\otimes T(I) \xrightarrow{\sg_{T(I),I}} T(T(I)\otimes I) \xrightarrow{T(r_{T(I)})} T(T(I)) \xrightarrow{\mu_I} T(I)$$with unit$$e_{T(I)}=\eta_I:I\to T(I).$$
\end{lem}

\begin{proof}
We have $T(I)\cong\uEe(I,U_TF_T(I))\cong \uEe^{T}(F_T(I),F_T(I))$ so that $T(I)$ inherits a monoid structure from the $\Ee$-enriched endomorphism object $\uEe^{T}(F_T(I),F_T(I))$.

In order to validate the asserted multiplicative map $m_{T(I)}$ and unit map $e_{T(I)}$, it suffices to show that $m_{T(I)}$ derives through adjunction from evaluation $$eval:\uEe^T(F_T(I),F_T(I))\otimes_\Ee F_T(I)\to F_T(I)$$or, what amounts to the same, from the composite map$$\uEe^T(F_T(I),F_T(I))\otimes U_T(F_T(I))\to U_T(\uEe^T(F_T(I),F_T(I))\otimes_\Ee F_T(I))\xrightarrow{U_T(eval)}U_TF_T(I)$$where the first map is the strength of the forgetful functor $U_T$, cf. Remark \ref{forgetfulstrength}.

This composite map in turn is adjoint to the composite map$$\uEe(I,T(I))\otimes T(I)\xrightarrow{\sg_{\uEe(I,T(I)),I}}T(\uEe(I,T(I))\otimes I)\xrightarrow{T(eval)}T(T(I))\xrightarrow{\mu_I}T(I)$$yielding precisely $m_{T(I)}$ up to the isomorphism $\uEe(I,T(I))\cong T(I)$ which identifies evaluation $\uEe(I,T(I))\otimes I\to T(I)$ with the unit constraint $T(I)\otimes I\cong T(I)$.\end{proof}

\begin{prop}\label{propMorStMon}
Any strong monad $(T,\mu,\eta,\sg)$ on a closed symmetric monoidal category $\Ee$ induces a strong monad transformation $\lambda:-\otimes T(I)\rightarrow T(-) $ given pointwise by
$$
\xymatrix @!0 @C=2.5cm @R=2cm{\relax
X\otimes T(I)    \ar[rr]^{\lambda_X} \ar[rd]_{\sg_{X,I}} && T(X)\\
& T(X\otimes I) \ar[ru]_{T(r_X)}
}
$$
\end{prop}

\begin{proof}
The natural transformation $\lambda$ is strong because the following diagram
$$
\xymatrix @C=4cm @R=2.3cm{\relax
X\otimes (Y\otimes T(I)) \ar[r]^{a_{X,Y,T(I)}} \ar[d]_{X\otimes \sg_{Y,I}} & (X\otimes Y)\otimes T(I)  \ar[d]^{\sg_{X\otimes Y,I}}\\
X\otimes T(Y\otimes I) \ar[r]^{T(a_{X,Y,I})\circ\sg_{X,Y\otimes I}} \ar[d]_{X\otimes T(r_Y) } & T((X\otimes Y)\otimes I) \ar[d]^{T(r_{X\otimes Y}) }\\
X\otimes T(Y) \ar[r]^{\sg_{X,Y}} & T(X\otimes Y)
}
$$commutes (the upper diagram commutes by the associativity constraint of the strength and the lower diagram commutes by the naturality of the strength).
It remains to be shown that $\lambda$ is a morphism of monads. The unit constraint of $\lambda$
%$$
%\xymatrix  @!0 @C=2cm @R=2cm{\relax
% X\otimes T\left(I \right)  \ar[rr]^{\lambda_X}  && T(X) \\
%& X \ar[ru]_{\eta_X} \ar[lu]^{\tilde{\eta}_X}
%}
%$$
follows from the commutativity of the following diagram
$$
\xymatrix @!0 @C=6cm @R=2cm{\relax
 X\otimes T(I)  \ar[r]^{\sg_{X,I}} & T(X\otimes I) \ar[r]^{T(r_X)} & T(X) \\
& X\otimes I \ar[u]_{\eta_{X\otimes I}} \ar[lu]_{X\otimes \eta_I} \ar[d]^{r_X} \\
& X \ar[ruu]_{\eta_X} \ar[luu]^{\tilde{\eta}_X}
}
$$
where the upper left triangle commutes by strong naturality of $\eta$, the lower left triangle by definition of $\tilde{\eta}$, and
the right ``square'' by naturality of $\eta$.

The multiplicative constraint of $\lambda$
$$
\xymatrix @!0 @C=5cm @R=2.5cm{\relax
(X\otimes T(I))\otimes T(I)   \ar[r]^{(\lambda\circ \lambda)_X} \ar[d]_{\tilde{\mu}_X} & T(T(X)) \ar[d]^{\mu_X} \\
 X\otimes T(I) \ar[r]^{\lambda_X} & T(X)
}
$$
follows from the commutativity of the following diagram
%$$
%\xymatrix @!0 @C=5cm @R=2.5cm{\relax
%\left(X\otimes T\left(I \right) \right)\otimes T\left( I\right)   \ar[rr]^{\lambda\circ \lambda} \ar[d]_{\tilde{\mu}} && TTX \ar[d]^{\mu} \\
% X\otimes T\left(I \right)  \ar[r]^{\sg} & T\left(X\otimes I\right) \ar[r]^{T\left(r\right) } & TX
%}
%$$
%and
$$
\xymatrix @!0 @C=5cm @R=2cm{\relax
(X\otimes T(I))\otimes T(I) \ar@{}[rrd]|{I}  \ar[rr]^{\sg_{X\otimes T(I),I}} \ar[d]_{\tilde{\mu}_X} && T((X\otimes T(I))\otimes I) \ar[d]^{T(r_{X\otimes T(I)})}   \\
 X\otimes T(I) \ar@{}[rrd]|{II} \ar[d]_{\sg_{X,I}} & X\otimes T^2(I) \ar[l]_{X\otimes \mu_I}  \ar[r]^{\sg_{X,T(I)}} & T(X\otimes T(I)) \ar[d]^{T(\sg_{X,I})}\\
T(X\otimes I) \ar[d]_{T(r_X)} \ar@{}[rrd]|{III} && T^2(X\otimes I) \ar[d]^{T^2(r_X)} \ar[ll]_{\mu_{X\otimes I}}\\
T(X) && T^2(X )\ar[ll]_{\mu_X}
}
$$in which Diagrams II and III commute by strong naturality of $\mu$, while Diagram I decomposes into three triangles
$$
\xymatrix @!0 @C=5cm @R=2cm{\relax
(X\otimes T(I))\otimes T(I) \ar[rd]^{(X\otimes\sg_{T(I),I})\circ a} \ar[rr]^{T(a)\circ\sg_{X\otimes T(I),I}} \ar[d]_{\tilde{\mu}_X} && T(X\otimes (T(I)\otimes I)) \ar[d]^{T(X\otimes r_{T(I)})}   \\
 X\otimes T(I) & X\otimes T(T(I)\otimes I) \ar[ru]^{\sg_{X,T(I)\otimes I}}\ar[l]_{X\otimes(\mu_I\circ T(r_{T(I)}))}  \ar[r]^{\sg_{X,T(I)}\circ(X\otimes T(r_{T(I)}))} & T(X\otimes T(I))
}
$$
where the left triangle commutes by Lemma \ref{propmonTI}, the middle triangle by the associativity constraint of $\sg$, and the right triangle by naturality of $\sg$.\end{proof}

\begin{cor}[cf. Proposition 1.9 in \cite{BMDCAO}]\label{propVSMMon}A strong monad $\,T$ on $\Ee$ is linear if and only if the monad transformation $\lambda:-\otimes T(I)\to T$ induces an equivalence between the category of $T$-algebras and the category of right $T(I)$-modules in $\Ee$.\end{cor}

\begin{proof}If $T$ is linear, $\lambda:-\otimes T(I)\to T$ is invertible and induces thus an equivalence. Conversely, if $\lambda^*$ is an equivalence, there exists a left adjoint $\lambda_!:\Mod_{T(I)}\to\Ee^T$ taking free $T(I)$-modules $X\otimes T(I)$ to free $T$-algebras $(T(X),\mu_X)$. The invertible unit of the $(\lambda_!,\lambda^*)$-adjunction at a free $T(I)$-module $X\otimes T(I)$ may then be identified with $\lambda_X$. This implies that the strength $\sg_{X,I}:X\otimes T(I)\to T(X\otimes I)$ is invertible for all $X$. The associativity constraint of $\sg$ shows finally that $\sg_{X,Y}$ is invertible for all $X,Y$, i.e. $T$ is linear.\end{proof}

\begin{cor}The category of monoids in $\Ee$ is equivalent to the category of linear monads on $\Ee$.\end{cor}

\begin{proof}The functor which takes a monoid $M$ to the linear monad $-\otimes M$ admits as quasi-inverse the functor which takes a linear monad $T$ to the monoid $T(I)$.\end{proof}

%\begin{rem}
In virtue of the preceding two corollaries the natural transformation of strong monads $\lambda:-\otimes T(I)\to T$ may be considered as a \emph{linear approximation} of the monad $T$. We conjecture that linear approximations exist for any strong monad on symmetric monoidal categories regardless of an existing closed structure. %This monadic interpretation of monoids is one possible way to understand the role of ``progenerators'' in the characterisation of abelian module categories, cf. the proof of Theorem \ref{thEnrGab} below.\end{rem}

\subsection{Gabriel's characterisation of categories of $R$-modules}\label{EnrCar}

We give a short proof of Gabriel's \cite{GDCA} characterisation of module categories among Grothendieck abelian categories using the formalism of strong monads. Recall that any additive category is canonically enriched in the closed symmetric monoidal category $\Ab$ of abelian groups. Additive functors are the same as $\Ab$-enriched functors.

By a module category we mean a category of $R$-modules for a unital associative ring $R$. Module categories are particular instances of \emph{Grothendieck abelian categories}, i.e. abelian categories which have filtered colimits commuting with finite limits and which have a generator. These categories are complete, cocomplete, wellpowered and co-wellpowered so that a functor between Grothendieck abelian categories admits a left (resp. right) adjoint if and only if the functor preserves limits (resp. colimits). In particular, any Grothendieck abelian category is \textit{tensored} over the closed symmetric monoidal category $\Ab$ of abelian groups (cf. Section \ref{Str}).

\begin{prop}\label{propTVSMon}For each object $P$ of finite type of a Grothendieck abelian category $\Cc$, the functor $\uHom_{\Cc}\left(P,-\right):\Cc \rightarrow \Ab $ and its left adjoint $-\otimes_\Ab P:\Ab\rightarrow\Cc$ induce a finitary linear monad on the category $\Ab$ of abelian groups.\end{prop}

\begin{proof}Both functors preserve filtered colimits and are $\Ab$-enriched so that they induce a finitary, strong monad. It remains to be proved that the strength of the induced monad $T(X)=\uHom_\Cc(P,X\otimes_\Ab P)$
$$\sg_{X,\ZZ}:X \otimes\uHom_{\Cc}(P,P)\rightarrow\uHom_{\Cc}(P,X \otimes_{\Ab}P)$$ is an isomorphism for each abelian group $X$.  Since domain and codomain of this morphism commute with filtered colimits in $X$, it suffices to establish the property for free abelian groups of finite rank. Let $X$ be a free abelian group of rank $n$. Then the left hand side may be identified with $\bigoplus_{n} \uHom_{\Cc}(P,P)$, the right hand side with $\uHom_{\Cc}(P,P\oplus\cdots\oplus P)$, and the strength itself with the canonical isomorphism $\bigoplus_{n} \uHom_{\Cc}(P,P)\cong\uHom_{\Cc}\left(P,P \oplus\cdots\oplus P\right)$. The monad $T$ is thus linear.\end{proof}

\begin{thm}[Gabriel \cite{GDCA}]\label{thEnrGab}A Grothendieck abelian category is equivalent to a category of $R$-modules if and only if it contains a projective generator $P$ of finite type. In this case, the ring $R$ may be chosen to be the endomorphism ring of $P$.\end{thm}

\begin{proof}Since $P$ is of finite type, we have a linear monad $T(X)=\uHom_{\Cc}(P,X \otimes_\Ab P)$ on the category of abelian groups by Proposition \ref{propTVSMon}. According to Corollary \ref{propVSMMon}, the category $\Ab^T$ of $T$-algebras is equivalent to the category of right $R$-modules where $R=T(\ZZ)=\uHom_\Cc(P,P)$. It suffices thus to show that the comparison functor $\phi:\Cc\to\Ab^T$ is an equivalence of categories. Since $\Cc$ is cocomplete and $P$ is projective, the functor $\uHom_\Cc(P,-):\Cc\to\Ab$ is right exact and the comparison functor $\phi$ has a fully faithful left adjoint. Since $P$ is a generator, the functor $\uHom_\Cc(P,-)$ is faithful, and hence so is $\phi$. Therefore, $\phi$ reflects monomorphisms and epimorphisms, and hence also isomorphisms. This shows that $\phi$ is conservative and admits a fully faithful left adjoint, i.e. $\phi$ is an equivalence of categories.\end{proof}

\begin{rem}Applying the preceding theorem to $\,\Cc=\Mod_S$ we obtain Morita's criterion \cite{MorDM} for when two rings $R,S$ have equivalent categories of modules: this is the case precisely when there exists an $R\!-\!S$-bimodule ${}_RP_S$ such that the right $S$-module $P_S$ is a projective generator of $\Mod_S$ of finite type fulfilling $R\cong\uHom_S(P_S,P_S)$.
\end{rem}

\section{Excisive model categories}\label{HMT}

In this section we introduce a class of pointed model categories, called \emph{excisive}, which roughly speaking capture half of the structure of a \emph{stable} model category. The category of based spaces, localised with respect to a generalised homology theory, is an example of an excisive model category which is not stable. The category of nonnegatively graded chain complexes over a ring (equipped with the projective model structure) is another such example. The category of $\Gamma$-spaces (equipped with the stable model structure of Bousfield-Friedlander) is a third such example.

Our main interest in excisive model categories comes from the fact that \emph{homotopy pushouts} therein can be characterised by having weakly equivalent parallel \emph{homotopy cofibres} in much the same way as pushouts in abelian categories can be characterised by having isomorphic parallel cokernels. This will be very helpful in establishing the main theorem of Section \ref{Mth}.

A model category is  \emph{left proper} (resp. \emph{right proper}) if weak equivalences are closed under pushout (resp. pullback) along cofibrations (resp. fibrations), cf. \cite{BouHT}.

\begin{defn}\label{excisive}A model category is said to be \emph{excisive} (resp. \emph{stably excisive}) if the category is pointed, the model structure is left proper, and the derived suspension functor is conservative (resp. fully faithful).

A model category is said to be \emph{coexcisive} (resp. \emph{stably coexcisive}) if the category is pointed, the model structure is right proper, and the derived loop functor is conservative (resp. fully faithful).\end{defn}

Recall (cf. Hovey \cite{HovMC}) that a \emph{stable model category} is a pointed model category such that the suspension-loop adjunction is a Quillen equivalence, i.e. induces a derived self-equivalence of the homotopy category. A proper pointed model category is thus a stable model category if and only if it is simultaneously stably excisive and stably coexcisive. Note that one of the two stability conditions may be dropped because for an adjoint pair of functors to be an adjoint equivalence it is enough to assume that one part is fully faithful and the other part is conservative.

Our terminological choice has been guided by the fact that the category of based spaces, localised with respect to a generalised homology theory $E$, is an example of an excisive model category. The derived suspension is conservative here because of the Suspension Isomorphism Theorem $E_\bullet(X,*)\cong E_{\bullet+1}(\Sigma X,*)$. The latter follows from \emph{homological excision} $E_\bullet(CX,X)\cong E_\bullet(CX/X,*)$ and the five-lemma.

Although excisiveness and coexcisiveness are dual concepts, they behave quite differently in practice. For instance, the category of simplicial groups (equipped with the Kan-Quillen model structure \cite{QuiHA}) is a model for the homotopy theory of based \emph{connected} spaces, and is easily seen to be a \emph{coexcisive} model category.

We fix for any object $X$ of our excisive model category $\Ee$ a factorisation of $X\to *$ into a cofibration $X\ito CX$ followed by a weak equivalence $CX\weq *$, and call $CX$ a \emph{cone} on $X$. We assume that $\Ee$ possesses a \emph{functorial cone construction}. This technical assumption is not absolutely necessary, but simplifies the arguments and is satisfied in all our examples. Accordingly, we define for any morphism $f:X\to Y$, the \emph{mapping cone} $C_f$ to be the pushout $C_f=Y\cup_XCX$. The \emph{suspension} $\Sigma X$ of $X$ is then defined by $\Sigma X=C_{(X\to *)}=CX/X$.

\begin{rem}\label{leftproper}The mapping cone $C_f$ is a homotopy cofibre of $f:X\to Y$. For a \emph{cofibration} $f:X\ito Y$ the canonical map $C_f\to Y/X$ is a weak equivalence if the model category is left proper. Indeed, consider the following commutative diagram
$$
\xymatrix @!0 @C=2.5cm @R=2cm{\relax
 X \ar@{ >->}[r] \ar@{ >->}[d]_f &  CX \ar[r]^\sim \ar@{ >->}[d] & \star \ar@{ >->}[d] \\
 Y \ar@{ >->}[r] &  C_f \ar[r] & Y/X
}
$$
in which the lower horizontal sequence composes to the quotient map $Y\to Y/X$. Left square and outer rectangle are pushouts so that the right square is a pushout as well and the induced map $C_f\to Y/X$ is a weak equivalence by left properness. It is therefore consistent to call the sequence $X\ito Y\to Y/X$ induced by a cofibration a \emph{cofibre sequence}. Without left properness much more effort is needed to get a homotopically invariant notion of cofibre sequence, cf. \cite{QuiHA, HovMC}.

Left properness also implies that the suspension $\Sigma X$ is weakly equivalent to the suspension $\Sigma(X_c)$ of any \emph{cofibrant replacement} $X_c\weq X$ of $X$. The derived suspension functor is therefore conservative if and only if $\phi:X\to Y$ is a weak equivalence whenever $\Sigma\phi:\Sigma X\to\Sigma Y$ is a weak equivalence. In other words, in an excisive model category the suspension functor is ``homotopically conservative''.

The derived suspension functor is fully faithful if and only if for any \emph{fibrant replacement} $\Sigma X\weq (\Sigma X)_f$ the adjoint $X\to\Omega((\Sigma X)_f)$ is a weak equivalence. Since fully faithful functors are conservative, any stably excisive model category is an excisive model category. The converse is not true, cf. Remark \ref{crosseffect}.\end{rem}

\begin{prop}\label{cofibre}A left proper pointed model category is excisive if and only if the following saturation property holds:\vspace{1ex}

For every natural transformation of cofibre sequences
\begin{diagram}{}\label{cofibre1}
$$
\xymatrix @!0 @C=2.5cm @R=2cm{\relax
 X \ar@{ >->}[r]^{f} \ar[d]_{\alpha} &  Y \ar[r]^{g} \ar[d]_{\beta} & Y/X \ar[d]_{\gamma} \\
 X' \ar@{ >->}[r]^{f'} &  Y' \ar[r]^{g'} & Y'/X'
}
$$
\end{diagram}
if two among $\alpha,\beta,\gamma$ are weak equivalences then so is the third.\end{prop}

\begin{proof}In any left proper pointed model category, if $\alpha$ and $\beta$ are weak equivalences then so is $\gamma$. Indeed, the quotients $Y/X$ and $Y'/X'$ are weakly equivalent to the homotopy cofibres $C_{f}$ and $C_{f'}$, and $\alpha,\beta$ induce a weak equivalence $C_f\to C_{f'}$ so that the $2$-out-of-$3$ property of weak equivalences allows us to conclude.

Assume now that the model category is excisive and that $\beta$ and $\gamma$ are weak equivalences. We get then the following natural transformation of cofibre sequences

\begin{diagram}{}\label{cofibre2}
$$
\xymatrix @!0 @C=2.5cm @R=2cm{\relax
 Y \ar@{ >->}[r] \ar[d]_{\beta} &  C_f \ar[r] \ar[d]_{\tilde{\gamma}} & C_f/Y =\Sigma X\ar[d]_{\Sigma\alpha} \\
 Y' \ar@{ >->}[r] &  C_{f'} \ar[r] & C_{f'}/Y'=\Sigma Y
}
$$
\end{diagram}
in which $\tilde{\gamma}$ is a weak equivalence by the $2$-out-$3$ property of weak equivalences, so that the map $\Sigma\alpha$ is a weak equivalence according to the preceding argument. The map $\alpha:X\to X'$ is therefore a weak equivalence as well, since the derived suspension functor is conservative. Let us finally assume that $\alpha$ and $\gamma$ in Diagram \ref{cofibre1} are weak equivalences. Then $\tilde{\gamma}$ and $\Sigma\alpha$ in Diagram \ref{cofibre2} are weak equivalences so that by the preceding proof $\beta:Y\to Y'$ is a weak equivalence as well.

Conversely, if the saturation property for cofibre sequences holds, then for any cofibration $X\ito Y$ we get the following natural transformation of cofibre sequences
$$
\xymatrix @!0 @C=2.5cm @R=2cm{\relax
 X \ar@{ >->}[r] \ar[d]_{f} &  CX \ar[r] \ar[d]^{\sim}_{Cf} & \Sigma X \ar[d]_{\Sigma f} \\
 Y \ar@{ >->}[r] &  CY \ar[r] & \Sigma Y
}
$$
so that $f$ is a weak equivalence if and only if $\Sigma f$ is, i.e. the derived supension functor is conservative.\end{proof}

\begin{prop}\label{cofibrecor}A left proper pointed model category is excisive if and only if the following two conditions are satisfied:
\begin{enumerate}
\item[(a)]A cofibration $X\ito Y$ is acyclic whenever its quotient $Y/X$ is.
\item[(b)]A map $h:Y\to Y'$ of cofibrations under $X$
$$
\xymatrix @C=1.5cm @R=1.2cm{\relax
& X\ar@{ >->}[ld] \ar@{ >->}[rd]\\
Y \ar[rr]^{h}  && Y'
}
$$
is a weak equivalence whenever its quotient map $h/X:Y/X\to Y'/X$ is.
\end{enumerate}\end{prop}

\begin{proof}If the model category is excisive then condition (a) follows from Proposition \ref{cofibre} by putting $\alpha=f$ and $\beta=id_Y$ in Diagram \ref{cofibre1}, and condition (b) follows from Proposition \ref{cofibre} by putting  $\alpha=id_X$ and $\beta=h$ in Diagram \ref{cofibre1}.

For the converse, observe first in order to show that the model category is excisive, it suffices to establish implication ($\beta,\gamma$ weqs $\implies$ $\alpha$ weq) in Diagram \ref{cofibre1}, cf. the proof of Proposition \ref{cofibre}. Let us thus assume that conditions (a) and (b) are satisfied and that $\beta,\gamma$ are weak equivalences.

Since $\gamma$ is a weak equivalence, condition (b) implies that the left square in Diagram \ref{cofibre1} is a homotopy pushout because the comparison map $X'\cup_XY\to Y'$ is a weak equivalence.  This homotopy pushout factors into two commuting squares by factoring $\alpha$ into a cofibration $\tilde{\alpha}:A\ito \tilde{A}$ followed by a weak equivalence $\tilde{A}\to A'$ and taking the pushout of $\tilde{\alpha}$ along the cofibration $f$. The weak equivalence $\beta$ factors then through the resulting cofibration $\tilde{\beta}:B\ito\tilde{B}$ via a weak equivalence $\tilde{B}\to B'$. Since $\beta$ is a weak equivalence, the $2$-out-of-$3$ property of weak equivalences implies that $\tilde{\beta}$ is a acyclic cofibration. Therefore, the induced weak equivalence on quotients $A'/A\to B'/B$ implies that $A'/A$ is acyclic, whence $\tilde{\alpha}$ is a acyclic cofibration by property (a) and $\alpha$ is a weak equivalence, as required.\end{proof}

\begin{rem}\label{charexc}The three implications in Proposition \ref{cofibre} play different roles: while implication ($\alpha,\beta$ weqs $\implies$ $\gamma$ weq) holds in any left proper pointed model category, implication ($\alpha,\gamma$ weqs $\implies$ $\beta$ weq) amounts to condition \ref{cofibrecor}(b), and implication ($\beta,\gamma$ weqs $\implies$ $\alpha$ weq) alone already implies that the model category is excisive.

Each pointed model category defines a \emph{Waldhausen category} \cite{Wald}, cf. Remark \ref{Brown}. By Remark \ref{leftproper} and Proposition \ref{cofibre}, for a left proper pointed model category to be excisive, is a property of the associated Waldhausen category. Implication ($\alpha,\gamma$ weq $\implies$ $\beta$ weq) holds, resp. condition \ref{cofibrecor}(b) is satisfied, if and only if the associated Waldhausen category satisfies the so-called \emph{extension property}, cf. \cite{Wald}.\end{rem}

\begin{prop}\label{propexc}In a left proper pointed model category fulfilling condition \ref{cofibrecor}(b) (e.g. any excisive model category) a commutative square is a homotopy pushout if and only if the induced map on parallel homotopy cofibres is a weak equivalence.\end{prop}

\begin{proof}In a left proper model category we can replace a commutative square$$
 \xymatrix @!0 @C=2cm @R=2cm{\relax
    A \ar[r]^f \ar[d]_\alpha & B \ar[d]^\beta \\
    C \ar[r]_g  & D \\
  }
$$with a weakly equivalent one in which $f$ and $g$ are cofibrations. It suffices thus to show that if $f$ and $g$ are cofibrations then the square is a homotopy pushout precisely when the induced map on quotients $B/A\to D/C$ is a weak equivalence. This quotient map factors as an isomorphism $B/A\cong (C\cup_AB)/C$ followed by a quotient map $(C\cup_AB)/C\to D/C$ under $C$. By means of condition \ref{cofibrecor}(b), the latter is a weak equivalence precisely when the comparison map $C\cup_AB\to D$ of the given square is a weak equivalence, which in left proper model categories is equivalent to the square being a homotopy pushout.\end{proof}

\begin{prop}\label{excisivereflection}Let $F:\Ee\to\Ee'$ be a functor between pointed model categories preserving finite colimits, cofibrations and weak equivalences.

\begin{itemize}\item[(a)]If $\Ee'$ is excisive and $F$ reflects weak equivalences, then $\Ee$ is excisive;\item[(b)]If in addition $\Ee'$ is stably excisive and $F$ is the left adjoint part of a Quillen adjunction with weakly invertible homotopy units, then $\Ee$ is stably excisive.\end{itemize}\end{prop}

\begin{proof}Left properness of $\Ee$ is inherited from $\Ee'$ since $F$ preserves cofibrations, weak equivalences and pushouts, and moreover reflects weak equivalences. Hypothesis (a) (resp. (b)) implies that the left derived functor $LF:\Ho(\Ee)\to\Ho(\Ee')$ is conservative (resp. fully faithful) and the left derived functor $L\Sigma_{\Ee'}:\Ho(\Ee')\to\Ho(\Ee')$ is also conservative (resp. fully faithful). It follows then from the commutation of left derived functors $L\Sigma_{\Ee'}\circ LF=LF\circ L\Sigma_\Ee$ that $L\Sigma_\Ee$ is conservative (resp. fully faithful) as well, i.e. $\Ee$ is an excisive (resp. stably excisive) model category.\end{proof}

Goodwillie \cite{Good} defines a functor to be excisive if the functor takes homotopy pushouts to homotopy pullbacks. Condition \ref{Goodwillie}(a) below expresses thus that the identity functor of the model category is excisive in Goodwillie's sense.

\begin{prop}\label{Goodwillie}For a left proper pointed model category the following conditions are equivalent:
\begin{itemize}\item[(a)]every homotopy pushout is a homotopy pullback;\item[(b)]every cofibre sequence is a homotopy fibre sequence.\end{itemize}Either of these conditions implies that the model category is stably excisive.\end{prop}

\begin{proof}Clearly (a) implies (b). Conversely, assume that (b) holds and consider the following commutative cube
$$
 \xymatrix {
    & X'  \ar@{ >->}[rr]^{f'} \ar@{ >-}[d] && Y' \ar@{ >->}[dd] \\
    X \ar[ru] \ar@{ >->}[rr]^f\ar@{ >->}[dd]  &\ar[d]& Y \ar[ru] \ar@{ >->}[dd] \\
    & CX' \ar@{-}[r]  &\ar[r]& C_{f'} \\
     CX \ar[ru]^\sim \ar[rr] && C_f \ar@{->}[ru]_{\sim} \\
  }
$$in which we assume that the top face is a homotopy pushout with a pair of cofibrations $f,f'$. By left properness we can always reduce to this case. The induced map on the homotopy cofibres is then a weak equivalence so that the bottom face is a homotopy pullback. Front and back faces are weakly equivalent to pushout squares of cofibre sequences and hence homotopy pullbacks by assumption. It follows that the top face is a homotopy pullback as well, whence (a).

Applying (b) to the cofibre sequence $X\ito CX\to\Sigma X$ yields a weak equivalence $X\to\Omega(\Sigma X)_f$ for each object $X$, whence the model category is stably excisive.\end{proof}

\begin{rem}Proposition \ref{Goodwillie} admits a partial converse due to an embedding theorem of Schwede \cite[Section 2.2]{SchSMCA}: every \emph{cofibrantly generated} stably excisive model category $\Ee$ embeds into its category of spectra $\Sp(\Ee)$ by a left Quillen functor $\Sigma^\infty$ in such a way that the homotopy units of the Quillen adjunction $\Sigma^\infty:\Ee\lrto\Sp(\Ee):\Omega^\infty$ are weakly invertible. Since $\Sp(\Ee)$ is a stable model category, this readily implies that every homotopy pushout in $\Ee$ is a homotopy pullback, i.e. the identity functor of $\Ee$ is excisive in Goodwillie's sense \cite{Good}.

Up to cofibrant generation our notion of stably excisive model category coincides thus with Goodwillie's notion of excisive identity functor.\end{rem}

\begin{rem}\label{crosseffect}The preceding remark helps to understand the difference between ``excisive'' and ``stably excisive'' model categories. Indeed, any excisive functor has ``vanishing cross-effects'', which in the case of an excisive identity functor means that the canonical map $X\vee Y\to X\tilde{\times} Y$ into the derived product is a weak equivalence. This is certainly not true for the excisive model category of based topological spaces, localised with respect to a generalised homology theory $E$, because homological excision yields here canonical isomorphisms $E_\bullet(X\times Y,X\vee Y)\cong E_\bullet(X\wedge Y,*)$.\end{rem}

\section{Homotopical Gabriel-Morita theory}\label{Mth}

In this central section we give sufficient conditions for a strong monad $T$ acting on a pointed monoidal model category $\Ee$ to the effect that the linear approximation of Section \ref{CEbSCEC} induces a Quillen equivalence $\lambda_!:\Mod_{T(I)}\lrto\Ee^T:\lambda^*$.

In a first ``crude'' version (cf. Theorem \ref{thMain}) we show that this is essentially the case if the strength of $T$ is weakly invertible and the forgetful functor $\Ee^T\to\Ee$ takes free cell attachments of $T$-algebras to homotopical cell attachments in $\Ee$.

In a second more useful version (cf. Theorem \ref{mainexcisive}) we show that if $\Ee$ is an excisive model category in the sense of Section \ref{HMT} then the condition on the forgetful functor can be replaced with a condition on the monad itself, namely that $T$ should take cofibre sequences to homotopy cofibre sequences. This second version relies on a careful homotopical analysis of the so-called bar resolution of a $T$-algebra. The more technical aspects of this analysis have been deferred to Section \ref{taubar}.

\subsection{Cell extensions and cell attachments}\label{SCgMc}

To ease and shorten terminology, a cofibration between cofibrant objects will be called a \emph{cell extension} throughout.

\begin{lem}[Gluing Lemma]\label{lemPatch}
Consider a commutative cube in a model category
$$
 \xymatrix {
    & X'  \ar[rr] \ar@{ >-}[d] && T' \ar@{ >->}[dd] \\
    X \ar[ru]^{\sim} \ar[rr]\ar@{ >->}[dd]  &\ar[d]& T \ar[ru]_{\sim} \ar@{ >->}[dd] \\
    & Y' \ar@{-}[r]  &\ar[r]& Z' \\
     Y \ar[ru]^{\sim} \ar[rr] && Z \ar@{.>}[ru]_{\sim} \\
  }
$$
such that front and back squares are pushouts, $X \ito Y$ and $X'\ito Y'$ are cell extensions and the three backward oriented arrows $X\weq X'$, $Y\weq Y'$, $T\weq T'$ are weak equivalences between cofibrant objects. Then the fourth backward oriented arrow $Z\rightarrow Z'$ is also a weak equivalence between cofibrant objects.\end{lem}

\begin{rem}\label{Brown}There are (at least) two proofs of the Gluing Lemma. The first uses the technique of Reedy categories (cf. \cite{Hi,HovMC}) and establishes the Gluing Lemma as a special instance of the fact that a weak equivalence between Reedy cofibrant diagrams with values in a model category induces a weak equivalence between their colimits. The second proof establishes the Gluing Lemma for any category of cofibrant objects in the sense of Kenneth Brown \cite{Brown} by a direct diagram chase.

A \emph{category of cofibrant objects} is a category equipped with an initial object, wide subcategories of cofibrations and weak equivalences fulfilling the following axioms\begin{enumerate}\item[(1)] Every object is cofibrant;\item[(2)] Weak equivalences have the $2$-out-of-$3$ property;\item[(3a)] Cofibrations are closed under cobase-change;\item[(3b)]Acyclic cofibrations are closed under cobase-change;\item[(4)]Every map factors as a cofibration followed by a weak equivalence.\end{enumerate}

If we assume axioms (1), (2), (3a) and (4) of a category of cofibrant objects then axiom (3b) is actually equivalent to Gluing Lemma \ref{lemPatch}. Indeed, one first shows that in virtue of Brown's Lemma, (3b) implies that factorisation (4) is stable under cobase-change along cofibrations. This implies the Gluing Lemma by a well-known cubical diagram chase. Conversely, (3b) is a special case of the Gluing Lemma.

A \emph{Waldhausen category} is a pointed category of cofibrant objects in which only axioms (1), (2), (3a) and the Gluing Lemma \ref{lemPatch} hold. Waldhausen \cite{Wald} treats axiom (4) separately and calls a functorial version of it the \emph{cylinder axiom}.\end{rem}

\begin{defn}\label{defcellatt}
A commutative square
$$
 \xymatrix @!0 @C=2cm @R=2cm{\relax
    X \ar[r] \ar@{ >->}[d] & T \ar[d] \\
    Y \ar[r]  & Z \\
  }
$$
will be called a \emph{homotopical cell attachment} if all objects of the square are cofibrant, the vertical map $X \ito Y$ is a cofibration (and hence a cell extension) and the comparison map $Y\cup_X T\to Z$ is a weak equivalence between cofibrant objects.
\end{defn}
\begin{rem}
If the comparison map is an isomorphism, we say that the pushout square is a \emph{cell attachment}. Observe that in a cell attachment the cofibrancy of the fourth object $Z$ is automatic, i.e. a property, while in a homotopical cell attachment it is part of the structure. The cofibrancy of the fourth object ensures that a homotopical cell attachment is a \emph{homotopy pushout} in the homotopical sense.
\end{rem}

\begin{lem}\label{lemReedGen}
The Gluing Lemma \ref{lemPatch} remains true if front and back squares of the cube are homotopical cell attachments.
\end{lem}
\begin{proof}
This follows from the 2-out-of-3 property of weak equivalences.
\end{proof}

\subsection{Free cell extensions and free cell attachments}\label{freeattach}--\vspace{1ex}

For a monad $T$ on a model category $\Ee$ permitting a transfer of the model structure to $\Ee^T$, a map of $T$-algebras $V\ito W$ will be called a \emph{free cell extension} if $V$ is a cofibrant $T$-algebra and $V\ito W$ is part of the following pushout square in $\Ee^T$
 $$
 \xymatrix @!0 @C=2cm @R=2cm{\relax
    F_T(X) \ar[r] \ar@{ >->}[d] & V \ar@{ >->}[d] \\
    F_T(Y) \ar[r]  & W \\
  }
$$where $X\ito Y$ is a cell extension in $\Ee$. Such a pushout square will be called a \emph{free cell attachment}. If $\Ee$ is pointed and $T$ preserves the null-object, the induced sequence $V\ito W\to W/V$ will be called a \emph{free cofibre sequence}. Notice that the quotient $W/V$ of this free cofibre sequence is isomorphic to the free $T$-algebra $F_T(Y/X)$.

\subsection{Monoidal model categories and tractability}\label{SMmc}--\vspace{1ex}

Recall that a \emph{monoidal model category} \cite{HovMC} is a closed symmetric monoidal category $(\Ee,\otimes,I)$ equipped with a compatible Quillen model structure. This compatibility is expressed by the \emph{pushout-product} and \emph{unit} axioms. The pushout-product axiom requires the monoidal structure $-\otimes-:\Ee\times\Ee\to\Ee$ to be a \emph{left Quillen bifunctor} \cite{HovMC}. The unit axiom is redundant if the monoidal unit is cofibrant.

We shall say that a monoidal model category is \emph{tractable} if the monoidal unit is cofibrant and the underlying model category is \emph{cofibrantly generated} admitting a generating set of cell extensions for its cofibrations, cf. Barwick \cite{BarLRModC}.

A monoid $(M,m,e)$ in $\Ee$ is \emph{well-pointed} if its unit map $e:I\to M$ is a cofibration in $\Ee$. Tensoring with a well-pointed monoid $M$ preserves (acyclic) cofibrations.

\begin{lem}[\cite{BMDCAO}, Proposition 2.7]\label{propmodmod}
In any cofibrantly generated monoidal model category the category of right modules over a well-pointed monoid carries a transferred model structure.\end{lem}

\begin{thm}\label{thMain}
Let $\left(T,\mu,\eta, \sg \right)$ be a strong monad on a tractable monoidal model category $(\Ee,\otimes,I)$ such that
\begin{enumerate}
 \item[(a)]the category of T-algebras admits a transferred model structure;
 \item[(b)]the unit $\eta_I:I\to T(I)$ at the monoidal unit is a cofibration;
 \item[(c)]the strength $\sg_{X,Y}: X\otimes T(Y) \rightarrow T(X \otimes Y)$ is a weak equivalence for all cofibrant objects $X,Y$ in $\Ee$;
 \item[(d)]the forgetful functor $U_T:\Ee^T\to\Ee$ takes free cell attachments in $\Ee^T$ to homotopical cell attachments in $\Ee$ (cf. Definition \ref{defcellatt} and Section \ref{freeattach}).
\end{enumerate}

Then the linear approximation $\lambda: -\otimes T(I)\rightarrow T$ induces a Quillen equivalence $(\lambda_!,\lambda^*)$ between the category of right $T(I)$-modules and the category of $T$-algebras and hence an adjoint equivalence of the corresponding homotopy categories
$$L\lambda_!:\Ho(\Mod_{T(I)})\simeq\Ho(\Ee^T):R\lambda^*.$$
\end{thm}
\begin{proof}By Corollary \ref{proptens}, the category $\Ee^T$ is a tensored $\Ee$-category. By Lemma \ref{propmodmod} and hypothesis $(b)$, the category of right $T(I)$-modules admits a transferred model structure, as does the category of $T$-algebras by hypothesis $(a)$. The linear approximation $\lambda:-\otimes T(I)\rightarrow T $ is a morphism of strong monads (cf. Proposition \ref{propMorStMon}) and induces thus an $\Ee$-adjunction. More precisely, the right adjoint $\lambda^*$ takes the $T$-algebra $(X,\xi_X:T(X)\rightarrow X)$ to the $T(I)$-module $(X,\xi_X\lambda_X:X\otimes T(I)\rightarrow X)$ and preserves the underlying objects. The left adjoint functor $\lambda_!$ exists since  $\Ee^{T}$ is cocomplete. The commutative diagram of right adjoint functors
$$
\xymatrix @!0 @C=3cm @R=1.8cm{\relax
\Ee^{T} \ar[rr]^{\lambda^*} \ar[rd]_{U_{T}} && Mod_{T(I)} \ar[ld]^{V_T} \\
& \Ee
}
$$induces an analogous commutative diagram of left adjoint functors, in particular $\lambda_!$ takes free $T(I)$-modules to free $T$-algebras.

Since both model structures are transferred, the right adjoint $\lambda^*$ preserves and reflects fibrations and weak equivalences, and is thus a right Quillen functor. It follows that $\lambda_!$ is a left Quillen functor and the adjoint pair $(\lambda_!,\lambda^*)$ is a Quillen pair. In order to prove that it is a Quillen equivalence it suffices now to show that for each cofibrant $T(I)$-module $M$, the unit of the adjunction $\eta_{M}: M\rightarrow \lambda^*\lambda_!(M)$ is a weak equivalence. The unit at a free module $X\otimes T(I)$ is given by the strength
$$\eta_{X\otimes T(I)}=\sg_{X,I}:X\otimes T(I)\to T(X)$$ and is therefore a weak equivalence if $X$ is cofibrant in $\Ee$, by hypothesis $(c)$ and the cofibrancy of the monoidal unit $I$. We shall now extend this property to all cofibrant $T(I)$-modules.

We first show that the property ``$\eta_Z:Z \to \lambda^*\lambda_!(Z)$ is a weak equivalence between $V_T$-cofibrant $T(I)$-modules'' is closed under cobase change of $Z$ along free $T(I)$-maps on cell extensions in $\Ee$. Let us consider the following commutative cube
$$
 \xymatrix {
    & T(X) \ar[rr] \ar@{ >-}[d]  && \lambda^*\lambda_!(Z') \ar[dd] \\
    X\otimes T(I)\ar[ru]^{\sim} \ar[rr] \ar@{ >->}[dd] &\ar[d]& Z' \ar[ru]_{\sim} \ar[dd] \\
    & T(Y) \ar@{-}[r]&\ar@{->}[r] & \lambda^*\lambda_!(Z) \\
     Y\otimes T(I) \ar[ru]^{\sim} \ar[rr] && Z \ar@{.>}[ru]_{\sim} \\
  }
$$
in which we suppose inductively that $\eta_{Z'}:Z' \to \lambda^*\lambda_!(Z')$ is a weak equivalence between $V_T$-cofibrant $T(I)$-modules and that $Z'$ is a cofibrant $T(I)$-module. Since $X$ and $Y$ are cofibrant in $\Ee$, the front square is a free cell attachment in $\Mod_{T(I)}$ in the sense of Definition \ref{defcellatt}. Since the forgetful functor $V_T:\Mod_{T(I)}\to\Ee$ preserves pushouts and cofibrant objects and detects weak equivalences, it suffices by Lemma \ref{lemReedGen} to prove that the back square is a homotopical cell attachment in $\Ee$ after application of $V_T$. This follows from hypothesis $(d)$ since the back square is the image under $\lambda^*$ of a free cell attachment in $\Ee^T$ and hence taken to a homotopical cell attachment in $\Ee$ by the forgetful functor $V_T$ because $V_T\lambda^*=U_T$. Therefore, the unit $\eta_Z:Z\to \lambda^*\lambda_!(Z)$ is a weak equivalence between $V_T$-cofibrant $T(I)$-modules and $Z$ is a cofibrant $T(I)$-module, as required for the inductive step.

Any cellular $T(I)$-module is obtained from the initial $T(I)$-module by a possibly transfinite composition of cobase changes of the aforementioned kind because $\Ee$ is cellularly generated. A well-known telescope lemma implies then that $\eta_Z$ is a weak equivalence for all cellular $T(I)$-modules. Any cofibrant $T(I)$-module is retract of a cellular one, so that $\eta_Z$ is a weak equivalence for all cofibrant $T(I)$-modules $Z$.\end{proof}

\begin{rem}\label{hypotheses}Theorem \ref{thMain} is a crude version of our homotopical Gabriel-Morita theory. The four hypotheses are different in nature:

Hypothesis $(a)$ can be reformulated if the monad $T$ preserves reflexive coequalisers in which case $\Ee^T$ is a cocomplete $\Ee$-category with $\Ee$-tensors and $\Ee$-cotensors. The latter yield path-objects for fibrant $T$-algebras so that the existence of a transferred model structure amounts to the existence of a \emph{fibrant replacement functor} for $T$-algebras provided Quillen's small object argument is available in $\Ee^T$ (cf. \cite{BMAHO}).

Hypothesis $(b)$ ensures the existence of a transferred model structure on $T(I)$-modules. Note that our proof of Theorem \ref{thMain} also uses the property that tensoring with $I\to T(I)$ takes cofibrant objects to cell extensions.

Hypothesis $(c)$ can be weakened and already follows from the special case $Y=I$ in virtue of the cofibrancy of the monoidal unit, the associativity constraint of the strength and Brown's Lemma. It is crucial that the strength $\sg_{X,I}$ coincides up to isomorphism with the unit of the $(\lambda_!,\lambda^*)$-adjunction at free $T(I)$-modules.

Hypothesis $(d)$ is certainly the most difficult to check in practice and has the inconvenience that it strongly involves the behaviour of pushouts of $T$-algebras. Section \ref{excellent} discusses conditions on $\Ee$ and on $T$ implying $(d)$.\end{rem}

\subsection{Excellent and homotopically right exact monads}\label{excellent}

This part of Section \ref{Mth} contains our main theorem. We freely use results and terminology of Section \ref{taubar}.

\begin{defn}\label{monadconditions}A strong monad $(T,\mu,\eta,\sg)$ on a monoidal model category $\Ee$ will be called \emph{excellent} if \begin{enumerate}\item the monad preserves filtered colimits and reflexive coequalisers;\item the unit $\eta_X:X\to T(X)$ is a cofibration at each cofibrant object $X$;\item the category of $\,T$-algebras admits a transferred model structure.\end{enumerate}\end{defn}

\begin{defn}\label{monadconditions2}A monad $(T,\mu,\eta)$ on a pointed model category $\Ee$ will be called \emph{homotopically right exact} if\begin{enumerate}\item the monad preserves the null-object;\item for each cell extension $X\ito Y$ the image $T(X)\to T(Y)$ is a cell extension and the induced map $T(Y)/T(X)\to T(Y/X)$ is a weak equivalence;\item the functor $U_T:\Ee^T\to\Ee$ takes free cell extensions to cell extensions.\end{enumerate}\end{defn}

\begin{prop}\label{appexcisive}Let $\,T$ be an excellent monad on a tractable monoidal model category $\Ee$ with standard system of simplices. If $\,T$ is homotopically right exact and $\Ee$ is excisive, then the forgetful functor $U_T:\Ee^T\to\Ee$ takes

\begin{enumerate}\item free cofibre sequences in $\Ee^T$ to homotopy cofibre sequences in $\Ee$;
\item free cell attachments in $\Ee^T$ to homotopical cell attachments in $\Ee$.\end{enumerate}\end{prop}

\begin{proof}We first show that (1) and (2) are equivalent. Consider a free cell attachment
$$
 \xymatrix @!0 @C=2cm @R=2cm{\relax
    F_T(X) \ar[r] \ar@{ >->}[d] & V \ar@{ >->}[d] \\
    F_T(Y) \ar[r]  & W \\
  }
$$where $X\ito Y$ is a cell extension in $\Ee$ and $V$ is a cofibrant $T$-algebra. In particular, the vertical quotients are isomorphic so that we get $F_T(Y/X)\cong W/V$ in $\Ee^T$. Since $T$ is homotopically right exact, the underlying maps of the two vertical arrows are cell extensions in $\Ee$. Therefore, since $\Ee$ is excisive, Proposition \ref{propexc} tells us that the underlying square in $\Ee$ is a \emph{homotopical cell attachment}, i.e. (2) holds, if and only if the induced map on quotients $$U_T(F_T(Y))/U_T(F_T(X))\to U_T(W)/U_T(V)$$ is a weak equivalence. Now (1) holds if and only if the canonical map $$U_T(W)/U_T(V)\to U_T(W/V)$$ is a weak equivalence. The composition of both yields the canonical map $$U_T(F_T(Y))/U_T(F_T(X))\to U_T(F_T(Y/X))\cong U_T(W/V)$$ which is a weak equivalence because $T$ is homotopically right exact. The $2$-out-of-$3$ property of weak equivalences shows that (1) and (2) are equivalent.

Let us establish (1). The simplicial bar resolution applied to the free cofibre sequence $V\ito W\to W/V$ defines a sequence $\Bb_.(V)\to\Bb_.(W)\to\Bb_.(W/V)$ of simplicial $T$-algebras. Since $\Ee$ is tractable, the transferred model structure on $\Ee^T$ possesses a generating system of cell extensions for its cofibrations, cf. Section \ref{SMmc}. This together with property (3) of a homotopically right exact monad implies (by a transfinite induction) that cofibrant $T$-algebras have an underlying cofibrant object. Therefore, by Theorem \ref{mainappendix}(a), in the following commutative diagram

$$
\xymatrix @!0 @C=4.5cm @R=2cm{\relax
 |U_T\Bb_.(W)/U_T\Bb_.(V)| \ar[r]^\alpha \ar[d]_{\sim} &  |U_T(\Bb_.(W)/\Bb_.(V))| \ar[r]^\beta  & |U_T\Bb_.(W/V)| \ar[d]^{\sim} \\
 U_T(W)/U_T(V) \ar[rr]_\gamma & & U_T(W/V)
}
$$
the vertical arrows are weak equivalences. In order to show that $\gamma$ is a weak equivalence, i.e. (1) holds, it is enough that $\alpha$ and $\beta$ are weak equivalences. The map $\beta$ underlies the weak equivalence $\Bb(W)/\Bb(V)\weq\Bb(W/V)$ of Corollary \ref{maincor}.

Since $U_T\Bb_.(W/V)$ is $\tau$-cofibrant by Theorem \ref{mainappendix}(c), Lemma \ref{tausaturation} just below implies that $U_T(\Bb_.(W)/\Bb_.(V))$ is $\tau$-cofibrant as well. Moreover, the canonical map $U_T\Bb_.(V)\to U_T\Bb.(W)$ is a degreewise cofibration between $\tau$-cofibrant simplicial objects, and hence a $\tau$-cofibration by Proposition \ref{tauBrown}(a). In particular, its quotient $U_T\Bb_.(W)/U_T\Bb_.(V)$ is $\tau$-cofibrant by Proposition \ref{tauBrown}(d). In virtue of Lemma \ref{tauleft} it is thus sufficient to show that $\alpha$ is the geometric realisation of a degreewise weak equivalence. This is indeed the case, since $T$ is homotopically right exact, and $$\alpha_n:U_T\Bb_n(W)/U_T\Bb_n(V)\to U_T(\Bb_n(W)/\Bb_n(V))$$ is of the form $T(Y)/T(X)\to T(Y/X)$ for  $X=T^nU_T(V)$ and $Y=T^nU_T(W)$.\end{proof}

\begin{lem}\label{tausaturation}In an excisive monoidal model category with standard system of simplices, for any two degreewise cofibrant simplicial objects which become weakly equivalent after geometric realisation, if one is $\tau$-cofibrant then so is the other.\end{lem}

\begin{proof}It suffices (by the $2$-out-of-$3$ property of weak equivalences) to show that for any map $\phi:X\to Y$ of degreewise cofibrant simplicial objects, for which $|X|\to|Y|$ is a weak equivalence, the fat realisation $\|X\|\to\|Y\|$ is a weak equivalence as well.

This is true for a degreewise weak equivalence $\phi$ because $\tau(\phi)$ is a degreewise weak equivalence between Reedy cofibrant objects in this case. Since any simplicial map factors as a Reedy cofibration followed by a degreewise weak equivalence, it remains to treat the case of a Reedy cofibration $\phi$. The geometric realisation $|\phi|:|X|\to|Y|$ is then an acyclic cofibration with acyclic quotient $|Y/X|$. Since $Y/X$ is Reedy cofibrant, it is $\tau$-cofibrant by Lemma \ref{taucof}, so that $\|Y/X\|$ is acyclic as well. Property \ref{cofibrecor}(a) of an excisive model category implies then that the fat realisation $\|X\|\to\|Y\|$ is an acyclic cofibration as well.\end{proof}

\begin{thm}\label{mainexcisive}Let $T$ be an excellent monad on a tractable monoidal model category $(\Ee,\otimes,I)$ with standard system of simplices.

If $\,T$ is homotopically right exact and $\,\Ee$ is excisive then the linear approximation $\lambda:-\otimes T(I)\to T$ induces a Quillen equivalence precisely when the strength $\sg_{X,Y}$ of the monad $\,T$ is a weak equivalence for all cofibrant objects $X,Y$ of $\Ee$.\end{thm}

\begin{proof}The excellence of $T$ implies hypotheses (a) and (b) of Theorem \ref{thMain}. Proposition \ref{appexcisive} implies that (d) also holds. Therefore if (c) holds then linear approximation is a Quillen equivalence. Conversely, we have seen in the proof of Theorem \ref{thMain} that the latter being a Quillen equivalence implies that $\sg_{X,I}$ and hence $\sg_{X,Y}$ are weak equivalences for all cofibrant objects $X,Y$, cf. Remark \ref{hypotheses}.\end{proof}

\subsection{Recovering a theorem of Schwede}Segal's \cite{SECCT} category of $\Gamma$-spaces will be denoted $\GS$. By \emph{$\Gamma$-space} we mean a functor $A:\Gamma^\op\to s\Set$ defined on finite based sets and taking values in simplicial sets. We require $\Gamma$-spaces to preserve singletons so that $\GS$ is a pointed category. According to Bousfield-Friedlander \cite{BouHT} the category $\GS$ carries a \emph{stable} model structure with the property that its homotopy category $\Ho(\GS)$ is stable homotopy of connective spectra. Despite of the terminology the category $\GS$ is \emph{not} a stable model category \cite{HovMC} but only a stably excisive model category in the sense of Section \ref{HMT} so that our results apply. In a genuin stable context other techniques are available, cf. Schwede-Shipley \cite{SchSMC}.

Each $\Gamma$-space $A$ induces a strong endofunctor $\uA$ of the category $s\Set_*$ of based simplicial sets (cf. \cite{SECCT,BouHT,BIWPSCILS}). Composition of endofunctors corresponds to a circle-product of $\Gamma$-spaces. Circle-monoids $(A,m:A\circ A\to A)$ correspond one-to-one to pointed simplicial Lawvere theories or, equivalently, to strong monads $(\uA,\underline{m}:\uA\circ\uA\to\uA)$ on $s\Set_*$ which are determined by their restriction to $\Gamma^\op$, cf. \cite{SchSHAT,Lawvere}.

Schwede shows in \cite[Theorem 4.4]{SchSHAT} that for any circle-monoid $(A,m:A\circ A\to A)$ connective stable homotopy of $(s\Set_*)^{\uA}$ is equivalent to the homotopy category of modules over a certain $\Gamma$-ring $\Gamma_A=(A,\mu:A\wedge A\to A)$. The smash-product of $\Gamma$-spaces is a closed symmetric monoidal structure on $\GS$ induced by Day convolution from the smash product of finite based sets, cf. Lydakis \cite{LydSG}. Smash-product and circle-product share the same unit $\Gamma^1=\Gamma(-,1_+)$, where $1_+$ is a two-element set.

The strong monad $\uA$ extends to a strong monad on $(\GS,\wedge,\Gamma^1)$ in such a way that homotopy of $\uA$-algebras in $\GS$ is connective stable homotopy of $\uA$-algebras in $s\Set_*$. On the other hand, Schwede's $\Gamma$-ring $\Gamma_A$ may be identified with the endomorphism-ring of $\uA(\Gamma^1)$. Therefore, Schwede's \cite[Theorem 4.4]{SchSHAT} becomes an instance of our Theorem \ref{mainexcisive} provided the hypotheses of the latter are satisfied.

The following three points must be verified:\begin{enumerate}\item[(a)]$(\GS,\wedge,\Gamma^1)$ is a monoidal model category;\item[(b)]The strength $\sigma_{B,C}$ of $\uA$ is a stable equivalence for cofibrant $\Gamma$-spaces $B,C$;\item[(c)]The monad $\uA:\GS\to\GS$ is excellent and homotopically right exact.\end{enumerate}

For this we use several times the fact that the \emph{assembly map} $B\wedge C\to B\circ C$ is a stable equivalence if $B$ or $C$ is a cofibrant $\Gamma$-space, as shown by Lydakis \cite[5.29]{LydSG}.

For (a) note that the pushout-product axiom holds for cofibrations (cf. \cite{LydSG,BMENRC}) so that it is enough to show that $B\wedge C$ is stably acyclic if $B$ or $C$ is. By Lydakis' theorem we can study the circle-product $B\circ C$ instead, where the property holds because a $\Gamma$-space is stably acyclic precisely when its associated Segal spectrum is.

For (b) note that the monad $\uA$ is given by left composition $A\circ-$ and the strength $\sg_{B,C}$ is given by a map $B\wedge (A\circ C)\to A\circ(B\wedge C)$ which for the unit $C=\Gamma^1$ coincides (up to a switch) with the assembly map. By Remark \ref{hypotheses} this suffices.

For (c) details may be found in \cite{RatMTEC} where the special case of a well-pointed circle-monoid $A$ is treated. The general case reduces to this thanks to \cite[Theorem C]{Rezk}. Let us mention that Condition \ref{monadconditions2}(2) for the monad $\uA$ follows from diagram:

$$
\xymatrix @!0 @C=3.5cm @R=1.75cm{\relax
 A\wedge B \ar[r] \ar[d]_{\sim} &  A\wedge C \ar[r] \ar[d]_{\sim} & A\wedge (C/B) \ar[d]_{\sim}\ar[rd]^\sim& \\
 A\circ B \ar[r] & A\circ C \ar[r] & (A\circ C)/ (A\circ B)\ar[r]& A\circ(C/B)
}
$$

\section{Fat realisation and simplicial bar resolution}\label{taubar}

We investigate homotopical properties of the \emph{simplicial bar resolution} of $T$-algebras for excellent monads $T$, cf. Definition \ref{monadconditions}. These properties have been essential in establishing Theorem \ref{mainexcisive}. %One of the major difficulties is the fact that, although the simplicial bar resolution of a $T$-algebra is often Reedy cofibrant, its underlying object (without the $T$-algebra structure) is rarely Reedy cofibrant. Much of this section is devoted to finding a convenient homotopical formalism to deal with this, cf. Theorem \ref{mainappendix}.
Our method relies on \cite[Appendix]{BerBoardVogtRes} where it has been shown that in presence of a \emph{standard system of simplices} there is a well behaved \emph{geometric realisation functor} for Reedy cofibrant simplicial objects in any monoidal model category. Thanks to Segal's \emph{fat realisation} \cite{SECCT} the good behaviour of the geometric realisation functor extends to a larger class of simplicial objects, half-way in between degreewise cofibrant and Reedy cofibrant.

We assume throughout that $\Cc$ is a model category and that the category $s\Cc$ of simplicial objects of $\Cc$ is endowed with the \emph{Reedy model structure}, cf. \cite{Hi,HovMC}. We also assume that $\Cc$ is an $\Ee$-model category for a monoidal model category $(\Ee,\otimes,I_\Ee)$ with cofibrant monoidal unit $I_\Ee$ and \emph{standard system of simplices} $\delta$.  This amounts to a cosimplicial object $\delta:\Delta\to\Ee$ such that left Kan extension along the Yoneda embedding $y:\Delta\to s\Set$ yields a symmetric monoidal \emph{left Quillen functor} $$|-|_\delta:(s\Set,\times,\Delta[0])\to(\Ee,\otimes,I_\Ee)$$with invertible unit-constraint $I_\Ee\cong|\Delta[0]|_\delta$ and weakly invertible structural maps $|D_1|_\delta\otimes|D_2|_\delta\weq|D_1\times D_2|_\delta$, cf. \cite[Proposition A.13]{BerBoardVogtRes}.

There is an associated \emph{geometric realisation functor} $|-|:s\Cc\to\Cc$ for the simplicial objects of $\Cc$ defined by the \emph{coend}$$|X|=X\otimes_\Delta\delta$$using the $\Ee$-action on $\Cc$. Since $y:\Delta\to s\Set$ is a Reedy cofibrant cosimplicial object of $s\Set$, the given cosimplicial object $\delta=|-|_\delta\circ y:\Delta\to\Ee$ is Reedy cofibrant and the geometric realisation functor $|-|:s\Cc\to\Cc$ is a left Quillen functor.

\subsection{Fat realisation}To the geometric realisation functor $|-|:s\Cc\to\Cc$ we associate a fat realisation functor $\|\!-\!\|:s\Cc\to\Cc$, following Segal \cite[Appendix A]{SECCT}. We first define an endofunctor $\tau:s\Cc\to s\Cc$ which serves as a Reedy cofibrant replacement functor for degreewise cofibrant simplicial objects. The endofunctor $\tau$ underlies the comonad induced by the following adjunction$$i^*:\Cc^{\Delta^\op}\lrto\Cc^{\Delta^\op_{inj}}:i_!$$ where $i:\Delta_{inj}\inc\Delta$ denotes the inclusion of the wide subcategory of \emph{injective} (i.e. face) operators. We put $\tau(X)=i_!i^*(X)$ with augmentation $\tau(X)\to X$ being the counit of the adjunction, and define the \emph{fat realisation} by $$\|X\|=|\tau(X)|\to|X|.$$Since the subcategory $\Delta_{inj}$ is a Reedy category with $\Delta_{inj}=\Delta_{inj}^+$, a presheaf on $\Delta_{inj}$ is Reedy cofibrant precisely when it is degreewise cofibrant. Each degreewise cofibrant simplicial object $X$ has thus a Reedy cofibrant restriction $i^*(X)$. Since the right adjoint $i^*$ is a right Quillen functor, the left adjoint $i_!$ is a left Quillen functor so that the endofunctor $\tau$ converts degreewise cofibrant into Reedy cofibrant simplicial objects. Explicitly, in simplicial degree $n$, we have the formula
$$\tau(X)_n=\coprod_{\Delta_{sur}([n],[k])}X_k$$ where the coproduct is taken over the finite set of \emph{surjective} (i.e. degeneracy) operators with domain $[n]$.

\begin{lem}\label{fat}Fat realisation takes degreewise (acyclic) cofibrations to (acyclic) cofibrations. In particular, it takes degreewise weak equivalences between degreewise cofibrant simplicial objects to weak equivalences between cofibrant objects.\end{lem}

\begin{proof}Geometric realisation takes (acyclic) Reedy cofibrations to (acyclic) cofibrations, and $\tau$ converts degreewise (acyclic) cofibrations into (acyclic) Reedy cofibrations. The second statement is a consequence of Brown's Lemma.\end{proof}

Reedy cofibrancy implies degreewise cofibrancy but the converse is wrong in general. On the other hand, fat realisation does not preserve finite limits as does geometric realisation, and is therefore less compatible with ``algebraic structures''. This motivates the following compromise between the two cofibrancy notions. %For the definition we need Rezk's notion of \emph{left equivalence}, i.e. a weak equivalence which remains a weak equivalence under cobase-change along cofibrations. In particular, we shall repeatedly use

%\begin{lem}[Rezk]\label{Rezk}--

%\begin{itemize}\item[(a)]Left equivalences have the $2$-out-of-$3$ property;\item[(b)]Inclusion and retraction of strong deformation retracts are left equivalences;\item[(c)]Weak equivalences between cofibrant objects are left equivalences.\end{itemize}\end{lem}

%\begin{proof}(a) This follows by combining the $2$-out-of-$3$ property of Quillen equivaleneces with the fact that $f:A\to B$ is a left equivalence if and only if the left Quillen functor $f_!:\Ee/A\to\Ee/B$ is a Quillen equivalence.

%(b) Strong deformation retracts $i:A\lrto B:r$ are closed under cobase-change along any map $A\to A'$ so that the inclusion $i$ is a left equivalence, cf. the proof of \cite[Proposition 2.4.9]{HovMC}. Therefore, by Lemma \ref{Rezk}(a), the retraction $r$ is a left equivalence as well.

%(c) By Brown's Lemma, any weak equivalence between cofibrant objects factors as an acyclic cofibration followed by a retraction of an acyclic cofibration. Since acyclic cofibrations are left equivalences, Lemma \ref{Rezk}(a) allows us to conclude.\end{proof}

%By definition, a model category is \emph{left proper} precisely when all weak equivalences are left equivalences so that left equivalences are useful only in model categories which are not left proper (e.g. in most of the frequently encountered transferred model structures). In these contexts a source of examples of left equivalences which are not of the form stated in Lemma \ref{Rezk}(c) is provided by the following

\begin{defn}A simplicial map $f:X\to Y$ is called \emph{$\tau$-cofibration} if it is degreewise a cofibration and the induced map $|X|\cup_{\|X\|}\|Y\|\to|Y|$ is a weak equivalence.\end{defn}

A simplicial object $X$ is thus \emph{$\tau$-cofibrant} precisely when $X$ is degreewise cofibrant and the map $\|X\|\to|X|$ is a weak equivalence. For any $\tau$-cofibrant simplicial object, fat realisation is thus a \emph{cofibrant replacement} of geometric realisation.

\begin{lem}\label{tauleft}Geometric realisation takes degreewise weak equivalences between $\tau$-cofibrant simplicial objects to weak equivalences.\end{lem}

\begin{proof}This follows from Lemma \ref{fat} and $2$-out-of-$3$ of weak equivalences.\end{proof}

\begin{lem}\label{biQuillen}Let $-\otimes-:\Aa\times\Bb\to\Cc$ be a left Quillen bifunctor (cf. \cite{HovMC}). For any cofibration $f:X\ito Y$ in $\Aa$ and any weak equivalence between cofibrant objects $g:V\weq Z$ in $\Bb$, the induced comparison map$$f\square g:X\otimes Z\cup_{X\otimes V}Y\otimes V\to Y\otimes Z$$is a weak equivalence in $\Cc$.\end{lem}

\begin{proof}Brown's Lemma yields a factorisation of $g:V\weq Z$ into an acyclic cofibration $h:V\overset{\sim}{\ito}W$ followed by a retraction $r:W\weq Z$ of an acyclic cofibration $i:Z\overset{\sim}{\ito}W$. This yields the following commutative diagram  in $\Cc$:

$$
\xymatrix @!0 @C=4cm @R=2cm{\relax
 X\otimes V \ar[r]^{X\otimes h} \ar[d]_{f\otimes V} &  X\otimes W \ar[r]^{X\otimes r} \ar[d]^{f\otimes W} & X\otimes Z \ar[d]^{f\otimes Z} \\
 Y\otimes V \ar[r]_{Y\otimes h} & Y\otimes W \ar[r]_{Y\otimes r} & Y\otimes Z
}
$$ The two squares have comparison maps $f\square h$ and $f\square r$. Let us denote by $\tilde{f}$ the pushout of the left vertical $f\otimes V$ along
$X\otimes h$ such that we get $(f\square h)\circ \tilde{f}=f\otimes W$, and by $\tilde{r}$ the pushout of $X\otimes r$ along $\tilde{f}$.

A diagram chase implies then that the comparison map $f\square g$ of the outer rectangle factors as the pushout of the comparison map $f\square h$ along $\tilde{r}$ followed by the comparison map $f\square r$. Since $f\square h$ is an acyclic cofibration it suffices thus to show that $f\square r$ is a weak equivalence. This is the case because a similar argument (applied to the maps $i,r$) shows that $f\square r$ is a retraction of a cobase-change of the acyclic cofibration $f\square i$, so that $f\square r$ a indeed a weak equivalence.\end{proof}

\begin{lem}\label{taucof}Any Reedy cofibration is a $\tau$-cofibration.\end{lem}
\begin{proof}There is a cosimplicial object $y_\tau:\Delta\to s\Set$ such that $\tau(X)$ becomes a coend $X\otimes_\Delta y_\tau$ and the augmentation $\tau(X)\to X$ a map of coends $X\otimes_\Delta y_\tau\to X\otimes_\Delta y$ where $y:\Delta\to s\Set$ is the Yoneda-embedding. Therefore, the map $|\tau(X)|=\|X\|\to|X|$ may be identified with a map of coends $X\otimes_\Delta|y_\tau|_\delta\to X\otimes_\Delta|y|_\delta$. The right hand side $|y_\tau|_\delta\to|y|_\delta$ is actually a weak equivalence of Reedy cofibrant cosimplicial objects of $\Ee$. Now, taking coends over $\Delta$ defines a \emph{left Quillen bifunctor}$$-\otimes_\Delta-:\Cc^{\Delta^\op}\times \Ee^{\Delta}\to\Cc$$where we take the Reedy model structures on simplicial, resp. cosimplicial objects. Lemma \ref{biQuillen} implies then that for a Reedy cofibration $f:X\to Y$ the induced map$$X\otimes_\Delta|y|_\delta\cup_{X\otimes_\Delta|y_\tau|_\delta}Y\otimes_\Delta|y_\tau|_\delta\to Y\otimes_\Delta|y|_\delta$$ is a weak equivalence, i.e. that $X\to Y$ is a $\tau$-cofibration.\end{proof}

\begin{prop}\label{tauBrown}If $\,\Cc$ is left proper, the following properties hold:
 \begin{itemize}\item[(a)]Any degreewise cofibration between $\tau$-cofibrant objects is a $\tau$-cofibration;\item[(b)]Any simplicial map between $\tau$-cofibrant objects factors as a Reedy cofibration followed by a degreewise weak equivalence between $\tau$-cofibrant objects;\item[(c)] The class of $\tau$-cofibrations is closed under composition;\item[(d)]The class of (acyclic) $\tau$-cofibrations between $\tau$-cofibrant objects is closed under cobase-change along simplicial maps between $\tau$-cofibrant objects.\end{itemize}\end{prop}

\begin{proof}(a) Any degreewise cofibration $f:X\to Y$ between degreewise cofibrant simplicial objects induces the following commutative square
$$
\xymatrix @!0 @C=1.8cm @R=1.8cm{\relax
   \|X\| \ar[r] \ar[d] & \|Y\|\ar[d] \\
   |X| \ar[r]  & |Y| \\
 }
$$in which the upper horizontal map is a cofibration between cofibrant objects by Lemma \ref{fat}. Since the vertical maps are weak equivalences by hypothesis, the comparison map $|X|\cup_{\|X\|}\|Y\|\to|Y|$ is a weak equivalence by left properness.

(b) Any map factors a Reedy cofibration followed by a degreewise weak equivalence. By Lemma \ref{taucof} the Reedy cofibration is a $\tau$-cofibration. Any $\tau$-cofibration with $\tau$-cofibrant domain has a $\tau$-cofibrant codomain, cf. (a).

(c) This follows by a diagram chase from Lemma \ref{fat} since weak equivalences are closed under cobase-change along cofibrations by left properness.

(d) It is sufficient by (b) to show closedness under cobase-change along Reedy cofibrations and along degreewise weak equivalences between $\tau$-cofibrant simplicial objects. The geometric realisation of an (acyclic) $\tau$-cofibration $X\to Y$ factors as an (acyclic) cofibration $|X|\to|X|\cup_{\|X\|}\|Y\|$ followed by a weak equivalence $|X|\cup_{\|X\|}\|Y\|\to|Y|$, cf. (a). Therefore, a cubical diagram chase implies that we just need to check that this (acyclic) cofibration/weak equivalence factorisation is preserved under the two aforementioned cobase-changes.

For a cobase-change along a Reedy cofibration this follows from left properness. For a cobase-change along a degreewise weak equivalence between $\tau$-cofibrant objects this follows from Lemmas \ref{fat} and \ref{tauleft}.\end{proof}

%\begin{rem}Proposition \ref{tauBrown} may be summarised by saying that for a left proper model category $\Cc$ the full subcategory of $s\Cc$ spanned by the $\tau$-cofibrant objects is a \emph{category of cofibrant objects} (cf. Remark \ref{Brown}) if we choose as cofibrations the $\tau$-cofibrations and as weak equivalences the degreewise weak equivalences.\end{rem}

\begin{defn}A simplicial object $X$ is said to be \emph{split-augmented} if $X$ comes equipped with extra-degeneracies $s_n:X_{n-1}\to X_n,\,n\geq 0,$ prolonging the usual simplicial identities, where $X_{-1}=\pi_0(X)$ and $\partial_0:X_0\to X_{-1}$ is the quotient map.

A simplicial object is said to be \emph{good} (cf. Segal \cite{SECCT}) if it is degreewise cofibrant and all degeneracies are cofibrations.\end{defn}

Under the hypotheses made at the beginning of Section \ref{taubar}, simplicial sets $D$ ``act'' on simplicial objects $X$ of $\Cc$ via the $\Ee$-action on $\Cc$. Explicitly, in simplicial degree $n$, we have $(X\otimes D)_n=\coprod_{D_n}X_n$ with obvious simplicial operators. In particular, the \emph{simplicial cylinder} on $X$ is defined by $\Cyl(X)=X\otimes\Delta[1]$.

\begin{lem}\label{good}The simplicial cylinder of a good simplicial object realises to a model-theoretical cylinder.\end{lem}

\begin{proof}We shall say that a simplicial map $X\to Y$ is \emph{cylindric} if the canonical map $Y\cup_X\Cyl(X)\to\Cyl(Y)$ is a Reedy cofibration whose geometric realisation is an acyclic cofibration. In particular, if the simplicial map $X\to Y$ is cylindric and $|\Cyl(X)|$ is a cylinder on $|X|$ then $|\Cyl(Y)|$ is a cylinder on $|Y|$.

Let $X$ be a good simplicial object. Since $sk_0(X)$ is a constant Reedy cofibrant simplicial object, it is enough to show that the simplicial maps $sk_{n-1}(X)\to sk_n(X)$ are cylindric for all $n>0$. Since the class of cylindric maps is closed under cobase-change, it suffices to show that $L_n(X)\otimes\Delta[n]\cup_{L_n(X)\otimes\partial\Delta[n]}X_n\otimes\partial\Delta[n]\to X_n\otimes\Delta[n]$ is cylindric. By a $2$-out-of-$3$ argument the latter follows from $L_n(X)\otimes\partial\Delta[n]\to L_n(X)\otimes\Delta[n]$ and $X_n\otimes\partial\Delta[n]\to X_n\otimes\Delta[n]$ being cylindric. This in turn is true since $\partial\Delta[n]\times\Delta[1]\cup\Delta[n]\times\Delta[0]\to\Delta[n]\times\Delta[1]$ is an acyclic cofibration of simplicial sets, and $X_n$ as well as $L_n(X)$ are cofibrant objects because $X$ is good.\end{proof}

\begin{prop}\label{propcof}Any split-augmented good simplicial object $X$ is $\tau$-cofibrant and induces weak equivalences $\|X\|\weq|X|\weq\pi_0(X).$\end{prop}

\begin{proof}We have a morphism of split-augmented simplicial objects
$$
 \xymatrix @!0 @C=1.8cm @R=1.8cm{\relax
    \tau(X) \ar[r] \ar[d] & X\ar[d] \\
    \pi_0(\tau(X)) \ar[r]  & \pi_0(X) \\
  }
$$
where the map on $_pi_0$ is invertible because $\tau(X)_1=X_0\sqcup X_1$ and $\tau(X)_0=X_0$ with face operators  $(\partial_0^{\tau_0(X)},\partial_1^{\tau_1(X)})=(id_{X_0}\sqcup\partial^X_0,id_{X_0}\sqcup\partial^X_1):\tau(X)_1\dto\tau(X)_0$ so that the coequaliser $\pi_0(\tau(X))$ of this pair equals the coequalier $\pi_0(X)$ of the pair $$(\partial^X_0,\partial^X_1):X_1\dto X_0.$$
We claim that the family $(s^X_{n+1}:X_n\to X_{n+1})_{n\geq 0}$ of extra-degeneracies for $X$ induces a family $(s^{\tau(X)}_{n+1}:\tau(X)_n\to \tau(X)_{n+1})_{n\geq 0}$ of extra-degeneracies for $\tau(X)$. Indeed, for $n=-1$, the extra-degeneracy $s^X_0$ induces an extra-degeneracy $$s^{\tau(X)}_0:\pi_0(\tau(X))\cong\pi_0(X)\xrightarrow{s_0^X} X_0=\tau(X)_0.$$For $n\geq 0$, on the summand of $\tau(X)_n$ indexed by the surjection $\phi:[n]\twoheadrightarrow[k]$, the extra-degeneracy $s^{\tau(X)}_{n+1}$ is defined by $s^X_{k+1}:X_k\to X_{k+1}$ landing in the summand of $\tau(X)_{n+1}$ indexed by the surjection $\bar{\phi}:[n+1]\twoheadrightarrow[k+1]$ where $\bar{\phi}(n+1)=k+1$ and $\bar{\phi}(i)=\phi(i)$ for $0\leq i\leq n$. The simplicial identities required for these extra-degeneracies $(s^{\tau(X)}_{n+1})_{n\geq 0}$ follow from the simplicial identities satisfied by $(s^X_{n+1})_{n\geq 0}$.

This implies that $X$, resp. $\tau(X)$ contains $\pi_0(X)$, resp. $\pi_0(\tau(X))$ as a constant simplicial deformation retract (cf. May \cite{MayGeoIterLoopSp}). Because $X$ is good and $\tau(X)$ is Reedy cofibrant and thus good, geometric realisation induces deformation retractions by Lemma \ref{good}, and hence weak equivalences $|X|\weq\pi_0(X)$ and $|\tau(X)|\weq \pi_0(\tau(X))$, as well as the asserted weak equivalence $|\tau(X)|\weq|X|$.\end{proof}

\subsection{The simplicial bar resolution}The following two statements will be the main application of this section. We shall use the so-called \emph{bar resolution}. The \emph{simplicial} bar resolution $\Bb.(A)$ of a $T$-algebra $A$ is a simplicial object in $\Ee^T$ which in degree $n$ is defined by the formula $\Bb_n(A)=(F_TU_T)^{n+1}(A)$ where $F_T:\Ee\lrto\Ee^T:U_T$ is the free-forgetful adjunction. The simplicial face operators are defined by$$\partial_i=(F_TU_T)^{n-i} \varepsilon_{(F_TU_T)^iA}:\Bb_n(A)\to \Bb_{n-1}(A) \text{ (for } 0\leq i\leq n \text{)}$$ and the simplicial degeneracy operators are defined by$$s_i=(F_TU_T)^{n-i}F_T\eta_{U_T(F_TU_T)^iA}:\Bb_n(A)\to\Bb_{n+1}(A) \text{ (for } 0\leq i\leq n\text{)}.$$
Here $\varepsilon:F_TU_T\Rightarrow id_{\Ee^T}$ is the counit of the free-forgetful adjunction. We shall denote the geometric realisation of $\Bb_.(A)$ by $\Bb(A)$ and call it simply the bar resolution of $A$. It is a $T$-algebra equipped with a canonical $T$-algebra augmentation $\Bb(A)\to A$.\vspace{1ex}

The forgetful functor takes the coequaliser presentation of $A$ to a split coequaliser
$$\xymatrix @!0 @C=3cm {\relax TTU_T(A) \ar@<2pt>[r]^{\mu_{U_TA}} \ar@<-2pt>[r]_{TU_T(\eps_{A}) } & TU_T(A) \ar@/_2pc/@{.>}[l]_{\eta_{TU_TA}} \ar[r]^{U_T(\eps_{A})} & U_T(A)\ar@/_2pc/@{.>}[l]_{\eta_{U_TA}}}$$
which actually provides the whole underlying simplicial object $U_T \Bb.(A)$ with extra-degeneracies  making it \emph{split-augmented} over $U_T(A)$. This implies that $U_T\Bb.(A)$ is a simplicial object containing the constant simplicial object $U_T(A)$ as a simplicial deformation retract (cf. May \cite{MayGeoIterLoopSp}). Therefore, Proposition \ref{propcof} shows that if the simplicial object $U_T(\Bb_.(A))$ is \emph{good} then geometric realisation yields a weak equivalence $|U_T(\Bb_.(A))|_\Ee\weq U_T(A).$ This observation plays a key role in the proof of the following theorem. Results similar to (a) and (b) below under slightly different hypotheses have been obtained by Johnson-Noel, cf. \cite[Proposition 3.23]{JohnsonNoel}. %The notion of \emph{excellent} monad has been introduced in Definition \ref{monadconditions}.

\begin{thm}\label{mainappendix}Let $T$ be an excellent monad on a cofibrantly generated monoidal model category $\Ee$ with cofibrant unit and standard system of simplices.\vspace{1ex}

\noindent For each $T$-algebra $A$ with cofibrant underlying object $U_T(A)$ the following holds:
\begin{enumerate}\item[(a)]the bar resolution $\Bb(A)\to A$ is a cofibrant replacement of $A$ in $\Ee^T$;\item[(b)]the simplicial bar resolution $\Bb_.(A)$ is Reedy cofibrant in $s\Ee^T$;\item[(c)]the underlying simplicial object $U_T(\Bb_.(A))$ is $\tau$-cofibrant in $s\Ee$.\end{enumerate}

\end{thm}

\begin{proof}Since $\,T$ is strong and preserves reflexive coequalisers, Corollary \ref{proptens} shows that the category of $T$-algebras is a tensored $\Ee$-category, the free-forgetful adjunction is strong, and the strength of the left adjoint $F_T:\Ee\to\Ee^T$ is invertible. Therefore, since the pushout-product axiom holds in $\Ee$, it holds for the $\Ee$-action on $\Ee^T$ when we restrict to the images under $F_T$ of the generating (acyclic) cofibrations of $\Ee$. By a transfinite induction %using the transferred model structure on $\Ee^T$, 
the pushout-product axiom holds then for all (acyclic) cofibrations of $\Ee^T$, whence $\Ee^T$ is an $\Ee$-model category.

The forgetful functor $U_T$ preserves the same colimits as the monad $T$ and hence all ``sifted'' colimits. Geometric realisation $|-|_{\Ee^T}:s\Ee^T\to\Ee^T$ is a sifted colimit (because the diagonal $\Delta\to\Delta\times\Delta$ is an initial functor, cf. Gabriel-Ulmer \cite{GabrielUlmer}) so that the canonical transformation $U_T\circ|-|_{\Ee^T}\to|-|_\Ee\circ U_T$ is invertible, i.e. for each simplicial $T$-algebra $X$ we get a canonical isomorphism $U_T(|X|_{\Ee^T})\cong|U_T(X)|_{\Ee}$.

Since the unit $\eta_E:E\to T(E)$ is a cofibration for each cofibrant object $E$ of $\Ee$, the monad $T$ preserves cofibrant objects. This and the cofibrancy of $U_T(A)$ imply that the underlying simplicial object $U_T(\Bb_.(A))$ is good and moreover split-augmented. Proposition \ref{propcof} implies then that $U_T(\Bb_.(A))$  is \emph{$\tau$-cofibrant} yielding (c). Moreover we have a weak equivalence $|U_T(B_.(A))|_\Ee\weq U_T(A)$ inducing a weak equivalence $|\Bb_.(A)|\weq A$ because $U_T:\Ee^T\to\Ee$ reflects weak equivalences. Since geometric realisation takes Reedy cofibrant to cofibrant objects, (a) follows now from (b).

In order to establish (b) we use the following universal property of the simplicial bar resolution, cf. Trimble \cite{Tri}: for any simplicial $T$-algebra $Y_.$ such that the underlying object $U_T(Y_.)$ is split-augmented, and any $T$-algebra map $A\to\pi_0(Y)$
$$
\xymatrix @!0 @C=1.8cm @R=1.8cm{\relax
   \Bb_.(A) \ar@{.>}[r]^{\exists !} \ar[d] & Y_.\ar[d] \\
   A \ar[r]  & \pi_0(Y_.) \\
 }
$$there exists a \emph{unique} simplicial map of $T$-algebras $\Bb_.(A)\to Y_.$ compatible with the extra-degeneracies. If we put $Y_.=\tau(\Bb_.(A))$, the identity on $A$ produces in this way a section $\Bb_.(A)\to\tau(\Bb_.(A))$ of the augmentation $\tau(\Bb_.(A))\to\Bb_.(A)$. Since $\Bb_.(A)$ is degreewise cofibrant, its $\tau$-resolution $\tau(\Bb_.(A))$ is Reedy cofibrant, so that $\Bb_.(A)$, as a retract of a Reedy cofibrant object, is itself Reedy cofibrant.\end{proof}

\begin{cor}\label{maincor}If in addition $\,T$ preserves null objects and cell extensions then the bar resolution functor $\Bb:\Ee^T\to\Ee^T$ takes free cell attachments to homotopical cell attachments, and free cofibre sequences to homotopy cofibre sequences.\end{cor}

\begin{proof}Since $T$ preserves cell extensions, for each cell extension $X\to Y$ in $\Ee$ the image $F_T(X)\to F_T(Y)$ has an underlying cell extension whence the simplicial bar resolution $\Bb_.(F_T(X))\to\Bb_.(F_T(Y))$ is degreewise a cell extension of $T$-algebras. Therefore, the $\tau$-resolution $\tau\Bb_.(F_T(X))\to\tau\Bb_.(F_T(Y))$ is a Reedy cofibration between Reedy cofibrant simplicial $T$-algebras, containing the given bar resolution $\Bb_.(F_T(X))\to\Bb_.(F_T(Y))$ as a retract, cf. the preceding proof. In particular, geometric realisation yields a cell extension of $T$-algebras $\Bb(F_T(X))\to\Bb(F_T(Y))$. The first claim follows now from Theorem \ref{mainappendix}(a) and Gluing Lemma \ref{lemPatch}.

For the second claim consider the following commutative diagram of $T$-algebras
$$
\xymatrix @!0 @C=2.3cm @R=1.8cm{\relax
 \Bb(F_T(X)) \ar[r] \ar@{ >->}[d] & \Bb(V) \ar[r] \ar@{ >->}[d] & \Bb(*) \ar@{ >->}[d] \\
 \Bb(F_T(Y)) \ar[r] &  \Bb(W) \ar[r] & \Bb(W/V)
}
$$induced by a free cell attachment. We have just seen that the left square and the outer rectangle are homotopical cell attachments. By a diagram chase the right square is a homotopical cell attachment as well, i.e. $\Bb(W)/\Bb(V)\weq\Bb(W/V)$.
\end{proof}


\begin{thebibliography}{99}

\bibitem{BarLRModC}C. Barwick, \emph{On left and right model categories and left and right {B}ousfield localizations}, Homology, Homotopy and Applications \textbf{12} (2010), 245--320.

\bibitem{BIWPSCILS}C. Berger, \emph{Iterated wreath product of the simplex category and iterated loop spaces}, Adv. Math. \textbf{213} (2007), 230--270.

\bibitem{BMAHO}C. Berger and I. Moerdijk, \emph{Axiomatic homotopy theory for operads}, Comment. Math. Helv. \textbf{78} (2003), 805--831.

\bibitem{BerBoardVogtRes}C. Berger and I. Moerdijk, \emph{The {B}oardman-{V}ogt resolution of operads in monoidal model categories}, Topology \textbf{45} (2006), 807--849.

\bibitem{BMDCAO}C. Berger and I. Moerdijk, \emph{On the derived category of an algebra over an operad}, Georgian Math. J. \textbf{16} (2009), 13--28.

\bibitem{BMENRC}C. Berger and I. Moerdijk, \emph{On an extension of the notion of {R}eedy category}, Math. Z. \textbf{269} (2011), 977--1004.

\bibitem{BorHCA} F. Borceux, \emph{Handbook of categorical algebra 2: Categories and structures}, Encycl. of Math. and Appl. \textbf{51} (1994), xvii+443pp.

\bibitem{BouHT}A. K. Bousfield and E. M. Friedlander, \emph{Homotopy theory of {$\Gamma $}-spaces, spectra, and bisimplicial sets}, Lect. Notes in Math. \textbf{658} (1978), 80--130.

\bibitem{Brown}K. Brown, \emph{Abstract homotopy theory and generalized sheaf cohomology}, Trans. Amer. Math. Soc. \textbf{186} (1973), 419--458.

%\bibitem{Day}B. Day, \emph{On closed categories of functors}, Lect. Notes in Math. \textbf{137} (1970), 1--38.

%\bibitem{DubucTriangle}E. Dubuc, \emph{Adjoint triangles}, Lect. Notes in Math. \textbf{61} (1968), 69--81.

\bibitem{GDCA}P. Gabriel, \emph{Des cat\'egories ab\'eliennes}, Bull. Soc. Math. France \textbf{90} (1962), 323--448.

\bibitem{GabrielUlmer}P. Gabriel and F. Ulmer, \emph{Lokal pr\"asentierbare Kategorien}, Lect. Notes in Math. \textbf{221} (1971).

\bibitem{Good}T. Goodwillie, \emph{Calculus III}, Geom. Topol. \textbf{7} (2003), 645--711.

\bibitem{Hi}P. S. Hirschhorn, \emph{Model categories and their localizations}, AMS Math. Surveys and Monographs \textbf{99} (2003), xvi+457pp.

\bibitem{HovMC}M. Hovey, \emph{Model categories}, AMS Math. Surveys and Monographs \textbf{63} (1999), xii+209pp.

\bibitem{JohnsonNoel}N. Johnson and J. Noel, \emph{Lifting homotopy $T$-algebra maps to strict maps}, Adv. Math. \textbf{264} (2014), 593--645.

\bibitem{KockMS}A. Kock, \emph{Monads on symmetric monoidal closed categories}, Arch. Math. \textbf{21} (1970), 1--10.

\bibitem{KockSF}A. Kock, \emph{Strong functors and monoidal monads}, Arch. Math. \textbf{23} (1972), 113--120.

\bibitem{Lawvere}F. W. Lawvere, \emph{Functorial semantics of algebraic theories}, Proc. Nat. Acad. Sci. \textbf{50} (1963), 869--872.

\bibitem{LydSG} M. Lydakis, \emph{Smash products and {$\Gamma$}-spaces}, Math. Proc. Cambridge Philos. Soc. \textbf{126} (1999), 311--328.

\bibitem{MayGeoIterLoopSp}J. P. May, \emph{The geometry of iterated loop spaces}, Lect. Notes in Math. \textbf{271} (1972), viii+175pp.

\bibitem{MogNCM}E. Moggi, \emph{Notions of computation and monads}, Selections from the 1989 IEEE Symposium on Logic in Computer Science, Inform. and Comput. \textbf{93} (1991), 55--92.

\bibitem{MorDM}K. Morita, \emph{Duality for modules and its applications to the theory of rings with minimum condition}, Sci. Rep. Tokyo Kyoiku Daigaku Sect. A \textbf{6} (1958), 83--142.

\bibitem{QuiHA}D. G. Quillen, \emph{Homotopical algebra}, Lect. Notes in Math. \textbf{43} (1967), iv+156pp.

\bibitem{RatMTEC}K. Ratkovic-Segrt, \emph{Morita theory in enriched context}, PhD, Nice 2012, arXiv:1302.2774.

\bibitem{Rezk}C. Rezk, \emph{Every homotopy theory of simplicial algebras admits a proper model}, Topology and Appl. \textbf{119} (2002), 65--94.

\bibitem{SchSMCA}S. Schwede, \emph{Spectra in model categories and applications to the algebraic cotangent complex}, J. Pure Appl. Algebra \textbf{120} (1997), 77--104.

%\bibitem{SStHomAlgGS}S. Schwede, \emph{Stable homotopical algebra and $\Gamma$-spaces}, Math. Proc. Cambridge Philos. Soc. \textbf{126} (1999), 329--356.

\bibitem{SchSHAT}S. Schwede, \emph{Stable homotopy of algebraic theories}, Topology \textbf{40} (2001), 1--41.

%\bibitem{SchAMMMC}S. Schwede and B. E. Shipley, \emph{Algebras and modules in monoidal model categories}, Proc. London Math. Soc. (3) \textbf{80} (2000), 491--511.

\bibitem{SchSMC}S. Schwede and B. E. Shipley, \emph{Stable model categories are categories of modules}, Topology \textbf{42} (2003), 103--153.

\bibitem{SECCT}G. Segal, \emph{Categories and cohomology theories}, Topology \textbf{13} (1974), 293--312.

\bibitem{StrFTM}R. Street, \emph{The formal theory of monads}, J. Pure Appl. Algebra \textbf{2} (1972), 149--168.

\bibitem{Tri}T. Trimble, \emph{On the Bar Construction}, $n$-Category Caf\'e, 31 may 2007.

\bibitem{Wald}F. Waldhausen, \emph{Algebraic K-theory of spaces}, Lect. Notes in Math. \textbf{1126} (1985), 318--419.

\end{thebibliography}
\end{document}